\documentclass[a4paper,11pt,reqno,twoside]{amsart}

\usepackage{amsmath}
\usepackage{amsfonts}
\usepackage{amssymb}
\usepackage{amsthm}
\usepackage{color}
\usepackage{ifpdf}
\usepackage{array}
\usepackage{url}
\usepackage{multirow,bigdelim}
\usepackage{ascmac}
\usepackage[all]{xy} 
\usepackage{graphicx}

\numberwithin{equation}{section}
\newtheorem{definition}{Definition}[section]
\newtheorem{theorem}{Theorem}[section]

\newtheorem{remark}{Remark}[section]

\newcommand*{\C}{\mathbb{C}}
\newcommand*{\R}{\mathbb{R}}
\newcommand*{\Z}{\mathbb{Z}}
\newcommand*{\N}{\mathbb{N}}

\newcommand{\comment}[1]{}
\title[Analytic theories around the simplest screw]%
      {Analytic theories around the simplest screw} 
\author[M. Suzuki]{Masatoshi Suzuki}
\date{Version of \today}
\AtBeginDocument{%
\begin{abstract}
We present several analytic theories related to screw functions and 
describe the connections among them, taking the screw function of 
the simplest screw line as a guiding example.
Note that this article does not contain any new results. 
Nevertheless, subsequent developments have suggested that the analytic structures 
associated with the simplest screw provide a useful prototype for 
a broader theory connected with zeta-functions, 
Weil's quadratic forms, and related Hilbert-space structures arising in analytic number theory. 
\end{abstract}
\maketitle
}
\begin{document}

\setcounter{tocdepth}{1}
\tableofcontents

%
\section{Introduction} 
%

This article surveys several analytic theories surrounding screw functions and
explains the connections among them through the example of the simplest screw line.
Screw functions arise naturally from isometric embeddings of the real line into Hilbert spaces
and provide a link between several areas of analysis.
The article is purely expository and does not contain any new results.

The theory of positive-definite kernels is a rich and fruitful area of analysis. 
In Kre\u{\i}n--Langer \cite{KrLa77,KrLa85,KrLa14}, screw functions,
Hilbert spaces defined by positive-definite kernels,
and their relation to the so-called canonical systems are studied in detail.
On the other hand,
as outlined in Winkler~\cite{Win14} and Woracek~\cite{Wo15},
canonical systems are closely related to de Branges spaces,
which are reproducing kernel Hilbert spaces consisting of entire functions
and play an important role in modern analysis.
We do not attempt to reproduce the proofs of these results.
Instead, our aim is to provide a concrete roadmap through these interconnected theories
by focusing on the screw function attached to the simplest screw line.
\medskip

What we call the ``simplest screw'' in this article is the one 
given by the mapping $x_0: \R \to \R^3$ defined by 
\begin{equation} \label{eq_101}
x_0(t) = (\cos(t),\,\sin(t),\,t). 
\end{equation}
A continuous map $x: \R \to \mathcal{H}$ 
from the real line to a {\it real} Hilbert space 
is called a screw line or helix  in $\mathcal{H}$ 
if the length of chord $\Vert x(t)-x(s) \Vert$ 
depends only on the difference $|t-s|$ 
for every $t,s \in \R$. 
The map \eqref{eq_101} is a screw line in this sense, 
because 
\[
\aligned 
\Vert x_0(t)-x_0(s) \Vert^2 
&= 2(1-\cos(t-s)) + (t-s)^2
\endaligned 
\]
with respect to the Euclidean metric on $\R^3$. 
Similarly, 
the map $x:\R \to \R^2$ defined by $x(t)=(\cos(t),\sin(t))$ is also a screw line, 
but we regard it as a two-dimensional ``spiral'' in a colloquial sense 
and do not call it 
``the simplest screw line" 
or 
``the simplest helix". 
Of course, this is just a play on words. 
For uniformity of terminology, 
following Kre\u{\i}n--Langer \cite{KrLa14} and von Neumann \cite{NeSc41}, 
we only use the term ``screw line'' rather than ``helix'' or ``spiral'' hereinafter
and call \eqref{eq_101} the simplest screw line.

A screw line $x:\R\to\mathcal H$
determines an isometric embedding of $\R$
into the Hilbert space $\mathcal H$.
Associated with this embedding is the function 
$F_x(t)=\Vert x(t)-x(0)\Vert$, 
called the screw function of $x$
by von Neumann--Schoenberg \cite{NeSc41}.
Various areas of analysis are connected through screw functions, 
but we do not aim to describe them exhaustively. 
In this article, we first discuss the connections among the following topics 
using the simplest screw line \eqref{eq_101}: 
\begin{enumerate}
\item Screw functions and screw lines; 
\item $L^2$-spaces $L^2(\tau)$ for the spectral measures $\tau$ of screw functions;    
\item Hilbert spaces $\mathcal{H}(G_g)$ defined by positive-definite kernels for screw functions; 
\item Model spaces and de Branges spaces. 
\end{enumerate}

We then turn to the following topics  
in order to state the relation between $\mathcal{H}(G_g)$ 
and de Branges spaces in more detail:
\begin{enumerate}
\item[(5)] Canonical systems of the first-order differential equations; 
\item[(6)] Structure Hamiltonians $H$ of de Branges spaces; 
\item[(7)] Hilbert spaces $\widehat{L}^2(H)$ defined by Hamiltonians.  
\end{enumerate}

In addition, we also mention the relations with the following topics as appendices:
\begin{enumerate}
\item[(8)] Krein's string; 
\item[(9)] Infinitely divisible distributions; 
\item[(10)] Mean periodic functions.  
\end{enumerate}

The sections that follow are structured so that 
the first half outlines a general theory about the topic of the section, 
and the second half applies it to the screw function 
of the simplest screw line \eqref{eq_101}. 
\medskip

Although the present article is devoted to the simplest screw and its associated analytic theories, the viewpoint developed here has recently found applications in analytic number theory. In particular, the simplest screw discussed throughout this article has served as a prototype for the study of screw functions associated with the Riemann zeta-function. These developments suggest that screw functions provide a natural bridge between harmonic analysis, operator theory, and analytic number theory. A brief account of these developments is included in Section~\ref{section_17}.

%
%
\section{Screw lines} \label{section_2} 
%
%

\subsection{} 

Following Masani~\cite{Ma72} and Ravaska~\cite{Ra80},
we define screw lines in a Hilbert space as follows.
The notion originates in the work of
Kolmogoroff~\cite{Ko40} and
von Neumann--Schoenberg~\cite{NeSc36,NeSc41}.
Let $\mathcal{H}$ be a complex Hilbert space 
with the inner product $\langle \cdot, \cdot \rangle$ 
and the norm $\Vert x \Vert=\sqrt{\langle x, x \rangle}$. 
The definition is also meaningful for real Hilbert spaces. 

\begin{definition} \label{def_201}
A continuous mapping $x(\cdot):\R \to \mathcal{H}$ 
is called a \textbf{\textit{screw line}} in $\mathcal{H}$ if the inner product 
\[
\langle x(b)-x(a),\,x(d)-x(c) \rangle
\]
is translation-invariant for all $a,b,c,d \in \R$, that is, 
\[
\langle x(b+t)-x(a+t),\,x(d+t)-x(c+t) \rangle
= \langle x(b)-x(a),\,x(d)-x(c) \rangle
\]
for all $t,a,b,c,d \in \R$. 
\end{definition}

If $x(\cdot)$ is a screw line,
then the chord length
$\Vert x(t)-x(s)\Vert$
depends only on $|t-s|$. 
The converse holds if $\mathcal{H}$ is a real Hilbert space, because
\[
\aligned 
\langle x(t+u)&-x(u) ,x(s+u)-x(u) \rangle \\
& = \frac{1}{2}\Vert x(t+u)-x(u) \Vert^2 +\frac{1}{2} \Vert x(s+u)-x(u) \Vert^2 \\
& \quad - \frac{1}{2} \Vert x(t+u)-x(s+u) \Vert^2  \\
&= \langle x(t)-x(0) ,x(s)-x(0) \rangle.
\endaligned 
\]
It is also easy to verify that 
$x(\cdot):\R \to \mathcal{H}$ is a screw line 
in the sense of Definition \ref{def_201} if and only if 
the inner product $\langle x(a+t)-x(t),\,x(b+t)-x(t) \rangle$ 
is independent of $t \in \R$ for all $a,b \in \R$. 
Kre\u{\i}n~\cite{Kr44}
generalized the notion of screw lines
by introducing screw arcs as follows.

\begin{definition} \label{def_202}
Let $I$ be an interval in $\R$. 
A continuous function $x(\cdot):I \to \mathcal{H}$ 
is called a \textbf{\textit{screw arc}} in $\mathcal{H}$ if 
\begin{equation} \label{eq_201}
B_x(t,s):=\langle x(t+u)-x(u),\,x(s+u) -x(u) \rangle
\end{equation}
is independent of $u \in I$ for all $t,s \in \R$ 
such that $t+u,s+u \in I$. 
We call the screw arc for $I=\R$ the \textbf{\textit{screw line}}.
\end{definition} 

The quantity $B_x(t,s)$ 
plays a central role in the subsequent theory.
Many analytic objects discussed later
can be recovered from this kernel.
See Kre\u{\i}n--Langer~\cite[Section 12]{KrLa14} for a brief description of screw arcs.  

\subsection{} 

The mapping \eqref{eq_101}
is a screw line
in the sense of Definition \ref{def_202},
since we obtain
\begin{equation} \label{eq_202}
B_{x_0}(t,s)=\cos(t-s)-\cos(t)-\cos(s) + 1 +st 
\end{equation}
by calculating the right-hand side of \eqref{eq_201} 
with respect to the Euclidean inner product. 

%
%
\section{Screw functions}
%
%

\subsection{} 

For $0<a \leq \infty$, 
following Kre\u{\i}n~\cite{Kr59} and Kre\u{\i}n--Langer~\cite{KrLa77}, 
we denote by $\mathcal{G}_a$ 
the set of all continuous functions $g$ on the interval $(-2a,2a)$ 
such that 
\[
g(-t) = \overline{g(t)}
\]
and the kernel
\begin{equation} \label{eq_301}
G_g(t,s) = g(t-s) - g(t) - g(-s) + g(0)
\end{equation}
is nonnegative definite on $(-a,a)$, that is,  
\begin{equation} \label{eq_302}
\sum_{i,j=1}^{n} G_g(t_i,t_j) \,  \xi_i \overline{\xi_j} \,\geq\, 0
\end{equation}
for all $n \in \N$, $\xi_i \in \C$, and $|t_i| < a$, $(i = 1, 2, . . . , n)$. 
In literature, a kernel satisfying \eqref{eq_302} is often referred 
to as a positive definite kernel or semi-positive definite kernel. 
However, in this note, we refer to such a kernel as a nonnegative definite kernel.
\medskip

The kernel $B_x(t,s)$ introduced in Section~\ref{section_2}
admits an intrinsic characterization in terms of positive-definite kernels.
The complex valued function $B(t,s)$ on $(-a,a)\times(-a,a)$, $0< a \leq \infty$, 
is equal to $B_x(t,s)$ of \eqref{eq_201} for some screw arc $x(\cdot)$ on $(-a,a)$ 
if and only if 
\[
B(t,s) = G_g(t,s), 
\quad s,t \in (-a,a)
\]
holds for some  $g \in \mathcal{G}_a$ satisfying $g(0) = 0$. 
(Such $g$ can be chosen so that $c=0$ in \eqref{eq_401} below.) 
Then it is uniquely determined 
by the screw arc $x$, 
and it determines this screw arc up to unitary equivalence. 
The proof of this fact is sketched in \cite[\S12.1]{KrLa14}. 
Motivated by this correspondence, 
we adopt the following definition of screw functions from \cite{KrLa14}.

\begin{definition} \label{def_301}
Let $0<a \leq \infty$. We call a member of $\mathcal{G}_a$ 
a \textbf{\textit{screw function}} on $(-2a,2a)$. 
Further, we call $x(\cdot): (-a,a) \to \mathcal{H}$ 
a screw line of a screw function $g \in \mathcal{G}_a$ if 
\[
G_g(t,s) = B_x(t,s)
\]
holds for $t,s \in (-a, a)$. 
\end{definition}

Hereinafter, the term ``screw function'' will be used in the sense of Definition \ref{def_301}. 

For a screw line $x(\cdot)$, 
we call 
\[
F_x(t) = \Vert x(t) - x(0) \Vert
\] 
the chordal length-function of $x$. 
In von Neumann--Schoenberg \cite{NeSc36,NeSc41},
a function $F$ is called a screw function
if it arises as the chordal length-function $F_x$ 
of a screw line $x$. 
In the case of real Hilbert spaces studied in \cite{NeSc36, NeSc41}, 
the screw functions in Definition \ref{def_301} are essentially the same 
as the chordal length-functions as follows. 

We have $F_x(t)=\sqrt{B_x(t,t)}$ by definition. 
If $x$ is a screw line in a real Hilbert space, 
then 
\begin{equation*} 
\aligned 
B_x(t,s) 
& = \frac{1}{2}B_x(t,t)+\frac{1}{2}B_x(s,s)-\frac{1}{2}B_x(t-s,t-s) \\
& =  \frac{1}{2}F_x(t)^2+\frac{1}{2}F_x(s)^2-\frac{1}{2}F_x(t-s)^2. 
\endaligned 
\end{equation*}
Therefore, $B_x(t,s)=G_g(t,s)$ if we set $g(t):= -F_x(t)^2/2$. 
Moreover, this $G_g(t,s)$ is nonnegative definite as noted in \cite[(1.3)]{NeSc41}. 
Hence, $g(t)= -F_x(t)^2/2$ is a screw function in the sense of Definition \ref{def_301}. 
\medskip

If $g \in \mathcal{G}_a$ and $0<b<a$, 
the restriction $g|_{(-2b,2b)}$ belongs to $\mathcal{G}_b$. 
Conversely, for $0<a<\infty$, each $g \in \mathcal{G}_a$ admits 
at least one continuation $\widetilde{g} \in \mathcal{G}_\infty$ 
\cite[Corollary 5.2]{KrLa14}. 
Moreover, a complete description of 
all continuations $\widetilde{g} \in \mathcal{G}_\infty$ 
for a given $g \in \mathcal{G}_a$ 
is known \cite[Theorem 5.4]{KrLa14}. 

\subsection{}  We define the function $g_0(t)$ on $\R$ by 
\begin{equation} \label{eq_303}
g_0(t)=-\frac{t^2}{2}+\cos(t)-1.
\end{equation}
Then, it is an even function and 
\[
\aligned 
G_{g_0}(t,s) 
& = \cos(t-s)-\cos(t)-\cos(s)+1+st \\
& = (\cos(t)-1)(\cos(s)-1)+\sin(t)\sin(s) + st. 
\endaligned 
\]
This kernel is nonnegative definite on $\R$, because 
\[
\aligned
\sum_{i,j=1}^{n} G_{g_0}(t_i,t_j) \,  \xi_i \overline{\xi_j}
&= \sum_{i,j=1}^{n}(\cos(t_i)-1)(\cos(t_j)-1) \,  \xi_i \overline{\xi_j} \\
& \quad + \sum_{i,j=1}^{n}\sin(t_i)\sin(t_j)\,  \xi_i \overline{\xi_j} 
+ \sum_{i,j=1}^{n}t_it_j \,  \xi_i \overline{\xi_j} \\
&= \left|\sum_{i=1}^{n}(\cos(t_i)-1) \,  \xi_i \right|^2
+ \left|\sum_{i=1}^{n} \sin(t_i) \,  \xi_i \right|^2
+ \left|\sum_{i=1}^{n} t_i \,  \xi_i \right|^2 \geq 0
\endaligned
\]
for any $t_i,t_j \in \R$ and $\xi_i,\xi_j \in \C$. 
Hence, $g_0(t)$ is a screw function on $\R$ 
in the sense of Definition \ref{def_301}. 
In addition, $g_0(t)$ is associated with the screw line \eqref{eq_101}, 
since 
\[
B_{x_0}(t,s)=G_{g_0}(t,s)
\]
 by \eqref{eq_202}. 

%
%
\section{$L^2(\tau)$: $L^2$ spaces for spectral measures $\tau$} 
%
%

\subsection{} 

The following theorem provides an integral representation for screw functions 
and introduces their spectral measures.

\begin{theorem} \label{thm_301}
Let $0<a \leq \infty$. 
The continuous function $g$ on $(-2a,2a)$ belongs to the class $\mathcal{G}_a$ 
if and only if it admits a representation
\begin{equation} \label{eq_401}
g(t)=g(0)+ic\, t + \int_{-\infty}^{\infty}\left(
e^{it\gamma}-1-\frac{it\gamma}{1+\gamma^2}
\right) \frac{d\tau(\gamma)}{\gamma^2}, \quad |t| < 2a
\end{equation}
with $c \in \R$ and a nonnegative measure $\tau$ such that
\begin{equation} \label{eq_402}
\int_{-\infty}^{\infty} \frac{d\tau(\gamma)}{1+\gamma^2} < \infty.
\end{equation}
For a given $g \in \mathcal{G}_a$, 
the measure $\tau$ in \eqref{eq_401} need not be unique. 
\end{theorem}
\begin{proof}
See \cite[Theorem 5.1]{KrLa14}. 
\end{proof}

\begin{definition} Let $g \in \mathcal{G}_a$, $0<a \leq \infty$. 
Each measure $\tau$ on $\R$ for which \eqref{eq_401} holds 
is called a \textbf{\textit{spectral measure}} of $g$. 
\end{definition}

For a given screw function $g \in \mathcal{G}_a$ with $g(0)=0$,
we can construct a screw arc satisfying
$G_g(t,s)=B_x(t,s)$
using a spectral measure of $g$.
Take a measure $\tau$ satisfying \eqref{eq_401}
and consider the Hilbert space $L^2(\tau)$.
Then the map $x:(-a,a) \to L^2(\tau)$ defined by the functions 
\[
x(t): \gamma~\mapsto~ \frac{e^{i\gamma t}-1}{\gamma}, \quad t \in (-a,a), ~\gamma \in \R
\]
is a screw arc on $(-a,a)$. In fact,  
$B_x(t,s) = \langle x(t+u) - x(u), x(s+u) - x(u) \rangle_{L^2(\tau)}$ 
is independent of $u \in (-a,a)$ as 
\[
\aligned 
B_x(t,s) 
& = \langle x(t+u) - x(u), x(s+u) - x(u) \rangle_{L^2(\tau)} \\
& = \int_{-\infty}^{\infty}
e^{i\gamma u}\frac{e^{i\gamma t}-1}{\gamma}\cdot e^{-i\gamma u}\frac{e^{-i\gamma s}-1}{\gamma} \, d\tau(\gamma)  \\ 
&= 
 \int_{-\infty}^{\infty}
\frac{e^{i\gamma t}-1}{\gamma}\,\frac{e^{-i\gamma s}-1}{\gamma} \, d\tau(\gamma)
\endaligned 
\]
for $s,t \in (-a,a)$ if $u,t+u,s+u \in (-a,a)$. 
A straightforward calculation using \eqref{eq_401} shows that $G_g(t,s)=B_x(t,s)$. 

\subsection{} For the screw function $g_0$ in \eqref{eq_303}, we have
\begin{equation*} 
\aligned 
g_0(t) 
&=-\frac{t^2}{2} + \frac{1}{2}\left(e^{it}-1-\frac{it}{2}\right)
+\frac{1}{2}\left(e^{-it}-1+\frac{it}{2} \right) \\
&= \lim_{\gamma \to 0}\left[\left(
e^{it\gamma}-1-\frac{it\gamma}{1+\gamma^2}
\right) \frac{1}{\gamma^2}\right] \\
& \quad + \frac{1}{2}\cdot \left( e^{it}-1-\frac{it}{1+1^2} \right) \frac{1}{1^2}
+ \frac{1}{2}\cdot \left( e^{-it}-1-\frac{-it}{1+(-1)^2} \right) \frac{1}{(-1)^2}. 
\endaligned 
\end{equation*}
Therefore, $g_0(t)$ can be written as \eqref{eq_401} using the measure 
\begin{equation} \label{eq_403} 
d\tau_0(\gamma) 
= \delta_0(\gamma) \, d\gamma 
+ \frac{1}{2}\, \delta_1(\gamma)\,d\gamma 
+ \frac{1}{2}\, \delta_{-1}(\gamma)\,d\gamma,
\end{equation}
where $\delta_a$ denotes the Dirac measure at $\gamma=a$ 
and $d\gamma$ is the Lebesgue measure. 
The family of functions 
\[
\R \ni t~\mapsto~x(t)=(it,e^{it}-1, -e^{-it}+1) \in L^2(\tau_0) \simeq \C^3
\]
gives a screw line $x$ satisfying $G_{g_0}(t,s)=B_x(t,s)$. 
This realization is unitarily equivalent to the original screw line \eqref{eq_101}.

%
%
\section{$\mathcal{H}(G_g)$: Hilbert spaces associated with screw functions} 
%
%

\subsection{} 

To each screw function $g$, we associate a Hilbert space
$\mathcal H(G_g)$ according to \cite[Section 5]{KrLa14}. 
Let $0<a \leq \infty$ and 
let $C_0(-a, a)$ be the space of all continuous functions on $(-a, a)$ 
which vanish in a neighborhood of the endpoints $\pm a$, that is, 
$\phi \in C_0(-a,a)$ means that there exists a compact set $K \subset (-a,a)$ 
such that $\phi(t)=0$ for $t \in (-a,a) \setminus K$. 

Let $g \in \mathcal{G}_a$ with $g(0)=0$. 
We denote by $\mathcal{L}(G_g)$ 
the subspace of $C_0(-a, a)$ consisting of all functions $\phi$ 
satisfying 
\[
\int_{-a}^{a} \phi(t) \, dt=0, 
\]
equipped with the inner product 
\[
\langle \phi_1, \phi_2 \rangle_{g,a} = \int_{-a}^{a}\int_{-a}^{a}G_g(t,s)\phi_1(s)\overline{\phi_2(t)}\,dsdt, 
\quad 
\phi_1, \phi_2 \in C_0(-a, a). 
\]
Let 
\[
\mathcal{L}^\circ(G_g)=
\{
\phi \in \mathcal{L}(G_g)~|~ \langle \phi, \phi \rangle_{g,a}=0
\}.
\]
Then, the quotient space $\mathcal{L}(G_g)/\mathcal{L}^\circ(G_g)$ forms 
a pre-Hilbert space by the norm 
\[
\Vert \phi + \mathcal{L}^\circ(G_g) \Vert_{g,a} 
:= \Vert \phi  \Vert_{g,a}. 
\]
This is well-defined 
because every
$\psi\in\mathcal L^\circ(G_g)$ 
has norm zero.
\comment{
This is well-defined, because 
\[
\Vert \phi + \psi  \Vert_{g,a} \leq \Vert \phi  \Vert_{g,a} + \Vert \psi  \Vert_{g,a}= \Vert \phi  \Vert_{g,a}
\] 
and 
\[
\Vert \phi + \psi  \Vert_{g,a} \geq |\Vert \phi  \Vert_{g,a} - \Vert - \psi  \Vert_{g,a}|= \Vert \phi  \Vert_{g,a}
\] 
for every $\psi \in  \mathcal{L}^\circ(G_g)$. 
}
(For a closed subspace $W$ of a Banach space $V$, 
the quotient norm is defined by $\Vert v + W \Vert_{V/W}:=\inf_{w \in W}\Vert v +w\Vert_V$.)

\begin{definition} 
For $g \in \mathcal{G}_a$, we let $\mathcal{H}(G_g)$ 
be the Hilbert space obtained by completing 
the pre-Hilbert space $\mathcal{L}(G_g)/\mathcal{L}^\circ(G_g)$.
\end{definition}
The linear space $\mathcal{L}^2(-a,a):=L^2(-a,a)/(\overline{L^2(-a,a)\cap \mathcal{L}^\circ(G_g)})$ 
is a (not necessarily closed) subspace of $\mathcal{H}(G_g)$, 
at least if $a$ is finite, since $G_g(t,s)$ is a Hilbert--Schmidt kernel on $(-a,a)$. 
Moreover, we see that $\mathcal{H}(G_g)$ is actually wider than $\mathcal{L}^2(-a,a)$.
Let $\{\delta_{\alpha,n}(t)\}_n \subset C_0(-a,a)$ 
be a sequence of continuous functions 
which converges to the Dirac distribution $\delta_\alpha(t)$ at $t=\alpha$ 
in the sense of distributions. 
Then, the sequence $\{\delta_{\alpha,n}(t)\}_n$ 
is a Cauchy sequence in $\mathcal{L}(G_g)$ 
by the continuity of $G_g(t,s)$.  
Hence, $\delta_\alpha(t)$ can be considered as an element of  $\mathcal{H}(G_g)$.  

Let $g \in \mathcal{G}_a$ and take a spectral measure $\tau$ of $g$ 
(cf. Theorem \ref{thm_301}). 
Then, the transform 
\[
\phi~\mapsto~\Phi_1(\phi,z):=\int_{-a}^{a} \phi(t) \frac{e^{izt} - 1}{z} \, dt
\]
defines a linear map from $\mathcal{L}(G_g)$ to $L^2(\tau)$ and satisfies 
\begin{equation} \label{eq_501}
\langle \phi_1,\phi_2 \rangle_{g,a}
= \int_{-\infty}^{\infty} \Phi_1(\phi_1,\gamma) \overline{\Phi_1(\phi_2,\gamma)} \, d\tau(\gamma), 
\quad \phi_1,\phi_2 \in \mathcal{L}(G_g).
\end{equation}
Therefore, the induced map $\Phi_1:\mathcal{H}(G_g) \to L^2(\tau)$ 
is an isometric embedding of a Hilbert space. 
The relation \eqref{eq_501} is equivalent to the representation 
\begin{equation} \label{eq_502}
G_g(t,s) = \int_{-\infty}^{\infty} 
\frac{e^{i\gamma t}-1}{\gamma}\,\frac{e^{-i\gamma s}-1}{\gamma} \, d\tau(\gamma).
\end{equation}
It is often useful to express the transform $\Phi_1$ as 
\[
\Phi_1(\phi,z) = \frac{\widehat{\phi}(z)-\widehat{\phi}(0)}{z}
\]
using the Fourier transform
\begin{equation} \label{eq_503}
\widehat{\phi}(z):=(\mathsf{F}\phi)(z):=\int_{-\infty}^{\infty} \phi(t) \, e^{izt} \, dt. 
\end{equation}
If $\phi \in \mathcal{L}(G_g)$, 
we have $\Phi_1(\phi,z)=\widehat{\phi}(z)/z$ by $\widehat{\phi}(0)=0$.  

\subsection{} 

Let $\mathfrak{D}(\mathsf{A})=\{ \phi \in \mathcal{L}(G_g)\,|\, \phi' \in \mathcal{L}(G_g)\}$ 
and define the linear operator $\mathsf{A}$ on $\mathcal{L}(G_g)$ by 
$\mathsf{A}\phi=i\phi'$ for $\phi \in \mathfrak{D}(\mathsf{A})$. 
The mapping $\Phi_1:\mathcal{L}(G_g)\times \R \to \C;$ 
$(\phi,\gamma) \mapsto \Phi_1(\phi,\gamma)$ 
is the directing functional of the operator $\mathsf{A}$.  
In particular, 
\[
\Phi_1(\mathsf{A}\phi,z) = z\, \Phi_1(\phi,z), \quad \phi \in \mathfrak{D}(\mathsf{A}).
\]
See \cite[Section 8]{GoGo97} and \cite[Section 1]{Wo17} for the theory of directing functionals.  
In $\mathcal{H}(G_g)$, 
$\mathsf{A}$ generates a densely defined symmetric operator, 
which we denote again by $\mathsf{A}$. 
There is a bijection between all minimal self-adjoint extensions 
$\widetilde{\mathsf{A}}$ of $\mathsf{A}$ in $\mathcal{H}(G_g)$ 
or in some larger Hilbert space $\widetilde{\mathcal{H}}$ 
and all measures $\widetilde{\tau}$ on 
$\R$ such that \eqref{eq_501} holds. 
The self-adjoint extension $\widetilde{\mathsf{A}}$ of $\mathsf{A}$, 
corresponding to $\widetilde{\tau}$, is orthogonal, 
that is $\widetilde{\mathcal{H}}=\mathcal{H}(G_g)$, 
if and only if the set of all functions 
$\Phi_1(\phi,\cdot)$, $\phi \in \mathcal{L}(G_g)$ 
is dense in $L^2(\widetilde{\tau})$, 
which is equivalent to the completeness of 
the set of all functions
\[
\gamma~\mapsto~\frac{e^{it\gamma}-1}{\gamma}, 
\quad \gamma \in \R, \quad |t|<a
\]
in $L^2(\widetilde{\tau})$. 

\subsection{} \label{section_5_3}

In the case of $g_0 \in \mathcal{G}_\infty$ in \eqref{eq_303}, 
\[
\Phi_1(\phi,0) 
= \lim_{\gamma \to 0} \int_{-\infty}^{\infty} \phi(t) \, \frac{e^{i\gamma t}-1}{\gamma} \,dt 
=  \int_{-\infty}^{\infty} it\,\phi(t) \, dt = (\widehat{\phi})'(0)
\]
\[
\Phi_1(\phi,\pm 1) 
= \pm \int_{-\infty}^{\infty} \phi(t) \, e^{\pm it} \,dt 
= \pm \widehat{\phi}(\pm 1)
\]
and 
\begin{equation} \label{eq_504}
\aligned 
\langle \phi,\phi \rangle_{g_0,\infty}
& = \int_{-\infty}^{\infty} 
\left| \Phi_1(\phi,\gamma) \right|^2 \, d\tau_0(\gamma)
 = \int_{-\infty}^{\infty} 
\left| \frac{\widehat{\phi}(\gamma)-\widehat{\phi}(0)}{\gamma} \right|^2 \, d\tau_0(\gamma) \\
& = 
\left| \lim_{\gamma \to 0}\frac{\widehat{\phi}(\gamma)-0}{\gamma} \right|^2 
+ \frac{1}{2} 
\left| \frac{\widehat{\phi}(1)-0}{1} \right|^2 
+ \frac{1}{2} 
\left| \frac{\widehat{\phi}(-1)-0}{-1} \right|^2
\endaligned 
\end{equation}
for $\phi \in \mathcal{L}(G_{g_0})$ by \eqref{eq_403}. 
Therefore, $\mathcal{H}(G_{g_0})$ is a three-dimensional space 
$\mathcal{L}(G_g)/{\rm Ker}(\Phi_1)$ (without completion) 
consisting of classes of $\phi$ such that 
\[
(\Phi_1(\phi,0),\Phi_1(\phi,1),\Phi_1(\phi,-1))=
((\widehat{\phi})'(0),\widehat{\phi}(1),\widehat{\phi}(-1))\not=(0,0,0).
\] 
Moreover, $\phi \mapsto \Phi_1(\phi,\cdot)$ provides an isometric isomorphism 
from $\mathcal{H}(G_{g_0})$ to $L^2(\tau_0)$ 
by \eqref{eq_403} and \eqref{eq_504}. 
Hence, the corresponding self-adjoint extension $\widetilde{\mathsf{A}}$ 
is orthogonal in $\mathcal{H}(G_{g_0})$ 
and has eigenvalues $0$, $\pm 1$. 

%
%
\section{Nevanlinna class}
%
%

\subsection{} 

In order to describe the Fourier transforms of screw functions, 
we recall the Nevanlinna class.

\begin{definition} 
We denote by $\mathcal{N}$ 
the class of all holomorphic functions on the upper half-plane $\C_+=\{z\,|\, \Im(z)>0\}$ 
which map $\C_+$ into $\C_+ \cup \R$. 
This class is called the \textbf{\textit{Nevanlinna class}}. 
\end{definition}

A function $Q(z)$ on $\C_+$ belongs to $\mathcal{N}$ if and only if 
it admits the formula 
\begin{equation} \label{eq_601}
Q(z) = az + b + \int_{-\infty}^{\infty}\left( \frac{1}{\gamma-z} - \frac{\gamma}{1+\gamma^2} 
\right) d\tau(\gamma)
\end{equation}
for $a \in \R_{\geq 0}$, $b \in \R$, and a measure $\tau$ on $\R$ satisfying 
\eqref{eq_402}. 
The measure $\tau$ is called a \textit{\textbf{spectral measure}} of $Q$. 
Moreover, 
\[
a = \lim_{z \to \infty} \frac{Q(z)}{z}, \qquad b = \Re(Q(i)).
\]

Kre\u{\i}n--Langer \cite[Satz 5.9]{KrLa77} 
proved the following relation 
between screw functions in $\mathcal{G}_\infty$ and functions of $\mathcal{N}$. 

\begin{theorem} \label{thm_6_1}
The equality 
\begin{equation} \label{eq_602}
\int_{0}^{\infty} g(t)\,e^{izt} \, dt = -\frac{i}{z^2}Q(z), \quad \Im(z) > h
\end{equation}
for some $h \geq 0$, depending on $g$ or $Q$,  
establishes a bijective correspondence 
between all functions $g \in \mathcal{G}_\infty$ with $g(0)=0$ 
and all functions $Q \in \mathcal{N}$ with the property 
\begin{equation} \label{eq_603}
\lim_{y \to +\infty} \frac{Q(iy)}{y} =0.
\end{equation}
\end{theorem} 

\subsection{} 

For the screw function of \eqref{eq_303}, we have
\[
\int_{0}^{\infty} g_0(t) e^{izt} \, dt = - \frac{i}{z^2} Q_0(z),\quad \Im(z)>0 
\]
with 
\begin{equation} \label{eq_604}
Q_0(z) = \frac{1-2z^2}{z(z-1)(z+1)}. 
\end{equation}
It is easy to verify that $Q_0(z)$ belongs to
the Nevanlinna class $\mathcal{N}$
from the partial fraction decomposition 
\[
Q_0(z) = -\frac{\!\!\bar{z}}{|z|^2} - \frac{\bar{z}-1}{2|z-1|^2} - \frac{\bar{z}+1}{2|z+1|^2}.
\]
The condition \eqref{eq_603} is clearly satisfied by $Q_0(z)$.
 
%
%
\section{Meromorphic inner functions and the Hermite--Biehler class} 
%
%

\subsection{} 

Each $Q \in \mathcal{N}$ extends to $\C_-$ via the relation $Q(\bar{z})=\overline{Q(z)}$, 
but its behavior on $\R$ is undefined. 
A complex-valued function is said to be {\it real} 
if it maps real numbers to real numbers on its domain. 
Following this terminology, we say that $Q \in \mathcal{N}$ is 
{\it real meromorphic} 
if $Q$ is a real meromorphic function on $\C$. 
If $Q \in \mathcal{N}$ is real meromorphic, 
it is related to a meromorphic inner function defined as follows.

\begin{definition}
A bounded holomorphic function $F$ on $\C_+$ 
such that 
\[
\lim_{y \to 0+}|F(x+iy)|=1
\] 
for almost all $x \in \R$ is called an \textbf{\textit{inner function}} in $\C_+$. 
If a meromorphic function on $\C$ is an inner function in $\C_+$, 
it is called a \textbf{\textit{meromorphic inner function}} (MIF, for short) in $\C_+$. 
We denote by $\mathcal{MI}$ the class of all meromorphic inner functions in $\C_+$. 
\end{definition}

For a real meromorphic $Q \in \mathcal{N}$, the linear fractional transformation 
\begin{equation} \label{eq_701}
\Theta(z) = \frac{i-Q(z)}{i+Q(z)}
\end{equation}
defines a meromorphic inner function in $\C_+$. 
Conversely, for a meromorphic inner function $\Theta$, 
the inverse linear fractional transformation 
\begin{equation} \label{eq_702}
Q(z) = i\, \frac{1-\Theta(z)}{1+\Theta(z)}
\end{equation}
defines a real meromorphic function in $\mathcal{N}$. 
Let $\mu_\Theta$ be the measure on $\R$ supported on the level set 
\[
\sigma(\Theta) = \{z \in \R ~|~\Theta(z)=-1\} 
\]
such that the point masses at $\gamma$ in the level set are 
\begin{equation} \label{eq_703}
\mu_\Theta(\gamma) = \frac{2\pi}{|\Theta'(\gamma)|}.
\end{equation}
Then,  
\[
Q(z) = a z + b + \int_{-\infty}^{\infty}\left( \frac{1}{\gamma-z} - \frac{\gamma}{1+\gamma^2} 
\right) \frac{d\mu_\Theta(\gamma)}{\pi}
\]
if $Q$ and $\Theta$ are related as \eqref{eq_701} or \eqref{eq_702}. 
Comparing with \eqref{eq_601}, 
\begin{equation} \label{eq_704}
d\tau(\gamma) = \frac{d\mu_\Theta(\gamma)}{\pi}. 
\end{equation}
Meromorphic inner functions correspond 
to entire functions of the Hermite--Biehler class 
defined as follows.  

\begin{definition} 
The \textbf{\textit{Hermite--Biehler class}}  
is the class of all entire functions $E$ satisfying 
\[
|E^\sharp(z)|<|E(z)| \quad \text{for $z \in \C_+$}
\]
and having NO real zeros, where $E^\sharp(z)=\overline{E(\bar{z})}$. 
We denote by $\mathcal{HB}$ the Hermite--Biehler class. 
\end{definition}

Although some literature allows real zeros for members of $\mathcal{HB}$, 
this note adopts the above definition. 
For an entire function $E \in \mathcal{HB}$, the quotient 
\begin{equation} \label{eq_705}
\Theta(z) = \frac{E^\sharp(z)}{E(z)}
\end{equation}
defines a meromorphic inner function. 
Conversely, if $\Theta(z)$ is a meromorphic inner function, 
there exists $E \in \mathcal{HB}$ 
such that \eqref{eq_705} holds \cite[\S2.3 and \S2.4]{HaMa03}. 
\medskip

As described above, 
the class of real meromorphic functions in $\mathcal{N}$, 
the class $\mathcal{MI}$ of meromorphic inner functions, 
and the class $\mathcal{HB}$ are in natural correspondence.

\subsection{} \label{section_7_2}
The function $Q_0$ in \eqref{eq_604} is a real meromorphic function.
The meromorphic inner function corresponding to $Q_0$ is 
\[
\Theta_0(z) = \frac{i-Q_0(z)}{i+Q_0(z)}= \frac{z^3-2iz^2-z+i}{z^3+2iz^2-z-i}. 
\]  
This can be written as $\Theta_0(z)=E_0^\sharp(z)/E_0(z)$ for
\begin{equation} \label{eq_706}
E_0(z):=z^3+2iz^2-z-i.
\end{equation}
We see that the polynomial $P(w):=iE_0(iw)=w^3+2w^2+w+1$ 
satisfies $P(0)=1$ and $P(-2)=-1$. 
Therefore, $P(w)$ has at least one root in the interval $(-2,0)$. 
Further, $P(w)$ has no other real roots.
Indeed,
$P'(w)=(3w+1)(w+1)$,
$P(-1)=1$,
and $P(-1/3)=23/27$.
Since both local extrema are positive,
$P(w)$ has exactly one real root. 
Hence $P(w)=(w-\alpha)(w-\beta)(w-\bar{\beta})$ 
for a real number $\alpha \in (-2,0)$ and a nonreal complex number $\beta$. 
The real part of $\beta$ is negative 
by $-2<\alpha<0$ and $2=-\alpha-(\beta+\bar{\beta})$. 
As a consequence, all roots of $E_0(z)$ lie in the lower half-plane $\C_-$: 
\[
E_0(z) = (z-\alpha_1)(z-\alpha_2)(z-\alpha_3), \quad \Im(\alpha_i)<0~(i=1,2,3). 
\]  
This factorization shows that $E_0(z)$ is an entire function in $\mathcal{HB}$, 
since $|(z-\bar{\alpha})/(z-\alpha)|<1$ for any $z \in \C_+$ 
if the imaginary part of $\alpha$ is negative.  
\medskip

From 
\begin{equation} \label{eq_707}
\frac{2\pi}{|\Theta_0^\prime(\gamma)|}
= 
\begin{cases} 
~\pi, & \gamma =0, \\
~\pi/2, & \gamma = \pm 1, \\
~0, & \text{otherwise} , 
\end{cases}
\end{equation}
together with the fact that $Q_0(i) = 3i/2$, 
we have the integral formula 
\[
Q_0(z) = \int_{-\infty}^{\infty}\left( \frac{1}{\gamma-z} - \frac{\gamma}{1+\gamma^2} 
\right) d\tau_0(\gamma)
\]
using the measure in \eqref{eq_403}.  

%
%
\section{Model subspaces and de Branges spaces} 
%
%
We review notions related to model spaces in the upper half-plane 
according to Garcia--Ross \cite[Section 10]{GaRo15}, 
Makarov--Poltoratski \cite[Section 1.2]{MaPo05}, 
and Havin--Mashreghi \cite[Section 2]{HaMa03}. 
(See also Garcia--Mashreghi--Ross \cite[Sections 3.6.3 and 5.10.4]{GMR}.) 
We also review the theory of de Branges spaces. 
The general theory of de Branges spaces is presented in the monograph \cite{dB68}, 
but proofs of many results are more accessible 
in de Branges' earlier papers \cite{MR0114002, MR0133455, dB61, MR0133457, MR0143016} 
and the thesis of Linghu \cite{Lin15b}. 
Linghu \cite{Lin15a} is also useful, as it summarizes 
the answers to some of the problems in the book \cite{dB68}. 
Winkler \cite{Win14} and Woracek \cite{Wo17}
are useful for reviewing the above theories together with related ones.
%
%
\subsection{} 
%
%

The Hardy space $H^2 = H^2(\C_+)$ for $\C_+$ 
is the space of all analytic functions $f$ on $\C_+$ endowed with the norm
\[
\Vert f \Vert_{H^2}^2 := \sup_{v>0} \int_{\R} |f(u+iv)|^2 \, du < \infty.
\] 
The Hardy space $\bar{H}^2 = H^2(\C_-)$ for the lower half-plane $\C_-$ is defined similarly. 
As usual, we identify $H^2$ and $\bar{H}^2$ with subspaces of $L^2(\R)$ 
via nontangential boundary values on the real line, 
$F \mapsto F|_{\R}$, so that $L^2(\R)=H^2 \oplus \bar{H}^2$. 
The Fourier transform \eqref{eq_503} 
provides an isometry of $L^2(\R)$ up to a constant 
such that $H^2 = {\mathsf F}L^2(0,\infty)$ and 
$\bar{H}^2 = {\mathsf F}L^2(-\infty,0)$ by the Paley-Wiener theorem. 

\begin{definition}
Let $\Theta$ be an inner function. 
The \textit{\textbf{model subspace}} (or \textit{\textbf{coinvariant subspace}}) 
$\mathcal{K}(\Theta)$ is the closed subspace of $H^2$ defined by
\begin{equation} \label{eq_801}
\mathcal{K}(\Theta)=H^2 \ominus \Theta H^2,
\end{equation}
where $\Theta H^2 = \{ \Theta(z)F(z) \, |\, F \in H^2\}$. 
\end{definition} 

The model subspace $\mathcal{K}(\Theta)$ has the alternative representation 
\begin{equation} \label{eq_802}
\mathcal{K}(\Theta) = H^2 \cap \Theta \bar{H}^2.
\end{equation}
If $\Theta$ is a meromorphic inner function, 
the restriction map $F \mapsto F|_{\sigma(\Theta)}$ 
is a unitary operator $\mathcal{K}(\Theta) \to L^2(\mu_\Theta)$ 
\cite[Theorem 2.1]{MaPo05}. 
On the other hand, for $E \in \mathcal{HB}$, the set 
\begin{equation} \label{eq_803} 
\mathcal{H}(E) := \{ f ~|~ \text{$f$ is entire, $f/E$ and $f^\sharp/E \in  H^2$} \}
\end{equation}
becomes a Hilbert space when equipped with the norm 
\[
\Vert f \Vert_{\mathcal{H}(E)} 
:= \left\Vert \frac{f}{E} \right\Vert_{H^2} 
= \left\Vert \frac{f}{E} \right\Vert_{L^2(\R)}.
\] 

\begin{definition} 
Let $E \in \mathcal{HB}$. 
The Hilbert space $\mathcal{H}(E)$ defined by \eqref{eq_803} 
is called the \textit{\textbf{de Branges space}} generated by $E$. 
If a linear subspace $V$ of $\mathcal{H}(E)$ 
can be written as $V=\mathcal{H}(\tilde{E})$ 
for some $\tilde{E} \in \mathcal{HB}$ 
and $\Vert f \Vert_{V}=\Vert f \Vert_{\mathcal{H}(E)}$ 
for every $f \in V$, 
then it is called a \textit{\textbf{de Branges subspace}} of $\mathcal{H}(E)$. 
\end{definition} 

If $\Theta$ is a meromorphic inner function such that $\Theta=E^\sharp/E$ 
for $E \in \mathcal{HB}$, 
the model subspace $\mathcal{K}(\Theta)$ is isomorphic and isometric to 
the de Branges space $\mathcal{H}(E)$ as a Hilbert space 
by $\mathcal{K}(\Theta) \to \mathcal{H}(E): \, F(z) \mapsto E(z)F(z)$. 
Therefore, $\mathcal{H}(E)$ is isometrically embedded into $L^2(\mu_\Theta)$ 
via the composition of the isometric isomorphism 
$\mathcal{K}(\Theta) \to \mathcal{H}(E)$
and the unitary restriction map $\mathcal{K}(\Theta) \to L^2(\mu_\Theta)$.

%
%
\subsection{} 
%
%

The de Branges space $\mathcal{H}(E)$ is a reproducing kernel Hilbert space 
with reproducing kernel
\begin{equation*} 
K(z,w) 
= \frac{\overline{E(z)}E(w) - \overline{E^\sharp(z)}E^\sharp(w)}{2\pi i (\bar{z} - w)} 
= \frac{\overline{A(z)}B(w) - A(w)\overline{B(z)}}{\pi (w - \bar{z})} \quad 
\quad (z,w \in \C)
\end{equation*}
(for $w = \bar{z}$ this formula has to be interpreted appropriately as a derivative), 
where 
\begin{equation} \label{eq_804}
A(z) := \frac{1}{2}(E(z)+E^\sharp(z)), \qquad  
B(z) := \frac{i}{2}(E(z)-E^\sharp(z)). 
\end{equation}
The reproducing formula 
\[
f(z)=\langle f, K(z,\cdot) \rangle_{\mathcal{H}(E)}=\langle f/E, K(z,\cdot)/E \rangle_{L^2(\R)}
\] 
for $f \in \mathcal{H}(E)$ and $z \in \C \setminus \R$ 
remains true for $z \in  \R$ if $\Theta=E^\sharp/E$ is analytic in a neighborhood of $z$. 
\medskip

If $E \in \mathcal{HB}$ is a polynomial of degree $n$, 
\[
\mathcal{H}(E) = \{ P(z) \in \C[z]\,|\, {\rm deg}\, P<n\} 
\]
as a set and $\mathcal{H}(E) \simeq \C^n$ as a vector space \cite[Example 1.11\,(i)]{Lin15b}. 
Using the moments
\[
m_k := \int_{-\infty}^{\infty} x^k \,d\mu_\Theta(x),\quad k=0,1,2,\dots,  
\]
the reproducing kernel of $\mathcal{H}(E)$ is given by 
\begin{equation} \label{eq_805}
K(z,w) = -\frac{1}{H_{n-1}} \det
\begin{bmatrix}
0 & 1 & w & \cdots & w^{n-1} \\
1 & m_0 & m_1 & \cdots & m_{n-1} \\
\bar{z} & m_1 & m_2 & \cdots & m_{n} \\
\vdots & \vdots & \vdots & & \vdots \\
\bar{z}^{n-1} & m_{n-1} & m_n & \cdots & m_{2n-2} \\
\end{bmatrix},
\end{equation}
where $H_n := \det (m_{i+j})_{i,j=0}^{n}$ \cite[Example 5]{Wo15}. 

%
%
\subsection{} 
%
%
De Branges spaces have the following axiomatic characterization. 

\begin{theorem}
A Hilbert space $\mathcal{H}$ consisting of entire functions is a de Branges space 
if and only if it satisfies the following three axioms: 
\begin{enumerate}
\item[(dB1)] For each $z \in \C\setminus \R$ the point evaluation 
$f \mapsto f(z)$ is a continuous linear functional on $\mathcal{H}$. 
\item[(dB2)] If $f \in \mathcal{H}$, $f^\sharp$ belongs to $\mathcal{H}$ 
and $\Vert f \Vert_{\mathcal{H}} = \Vert f^\sharp \Vert_{\mathcal{H}}$. 
\item[(dB3)] If $w \in \C \setminus \R$, $f \in \mathcal{H}$ and $f(w)=0$, 
\begin{equation*}
\frac{z-\bar{w}}{z-w}f(z) \in \mathcal{H} \quad \text{and} \quad 
\left\Vert \frac{z-\bar{w}}{z-w}f(z) \right\Vert_{\mathcal{H}} = \Vert f \Vert_{\mathcal{H}}.
\end{equation*}
\end{enumerate}
\end{theorem}

If $\mathcal{H}$ satisfies the above axioms, 
there exists $E \in \mathcal{HB}$ 
such that $\mathcal{H}=\mathcal{H}(E)$ and 
$\Vert f \Vert_{\mathcal{H}}=\Vert f \Vert_{\mathcal{H}(E)}$ for all $f \in \mathcal{H}$. 
The choice of $E$ is not unique. 
In fact, $\mathcal{H}(E)=\mathcal{H}(E_\theta)$ 
for any $\theta\in [0,\pi)$ if  $E_\theta(z):=e^{i\theta}E(z)$.  

\subsection{} It is known that $\mathcal{H}(E_1)=\mathcal{H}(E_2)$ if and only if 
\[
\begin{bmatrix} A_2(z) \\ B_2(z) \end{bmatrix}
= 
\begin{bmatrix} a & b \\ c & d \end{bmatrix}
\begin{bmatrix} A_1(z) \\ B_1(z) \end{bmatrix}\quad \text{for some} \quad  
\begin{bmatrix} a & b \\ c & d \end{bmatrix} \in SL_2(\R). 
\]
In fact, we see that this relation does not change the reproducing kernel
\[
K(z,w) 
= \frac{\overline{A_1(z)}B_1(w) - A_1(w)\overline{B_1(z)}}{\pi (w - \bar{z})}
= \frac{\overline{A_2(z)}B_2(w) - A_2(w)\overline{B_2(z)}}{\pi (w - \bar{z})}. 
\]
However, it does change $Q$ 
\[
Q_1(z) = i\, \frac{1-\Theta_1(z)}{1+\Theta_1(z)} = \frac{B_1(z)}{A_1(z)}, 
\quad 
Q_2(z) = i\, \frac{1-\Theta_2(z)}{1+\Theta_2(z)} = \frac{B_2(z)}{A_2(z)},
\]
and therefore it also changes $\Theta$. 
Moreover, $\mathcal{K}(\Theta_1)\not=\mathcal{K}(\Theta_2)$ 
even if  $\mathcal{K}(\Theta_1)\simeq \mathcal{K}(\Theta_2)$ as a Hilbert space.  

Here, we recall the Iwasawa decomposition 
\[
\begin{bmatrix} a & b \\ c & d \end{bmatrix}
= n(x)a(r)\kappa(\theta)
=
\begin{bmatrix} 1 & 0 \\ x & 1 \end{bmatrix} 
\begin{bmatrix} r & 0 \\ 0 & r^{-1} \end{bmatrix}
\begin{bmatrix}  \cos\theta & \sin\theta \\ -\sin\theta & \cos\theta \end{bmatrix}. 
\]
Then $Q$ is invariant under the transformations $n(x)$ and $a(r)$.  
If we set 
\[
\begin{bmatrix} A_\theta(z) \\ B_\theta(z) \end{bmatrix}
= 
\begin{bmatrix}  \cos\theta & \sin\theta \\ -\sin\theta & \cos\theta \end{bmatrix}
\begin{bmatrix} A(z) \\ B(z) \end{bmatrix},  
\]
we have 
\[
E_\theta := A_\theta -i B_\theta = e^{i\theta}(A-iB) = e^{i\theta} E. 
\]
Therefore, $Q_\theta$ and the associated measure $\mu_{\Theta_\theta}$ 
are parametrized by $\theta \in [0,\pi)$. 
As we will see in Section \ref{section_9_1}, 
these measures correspond to the self-adjoint extensions of 
the multiplication operator by the independent variable.

\subsection{} As found in Section \ref{section_7_2}, 
the polynomial $E_0$ in \eqref{eq_706} belongs to $\mathcal{HB}$. 
Therefore, $\mathcal{H}(E_0)$ is well defined and is three-dimensional, 
since the degree of $E_0$ is three. We have 
\[
A_0(z) = \frac{1}{2}(E_0(z)+E_0^\sharp(z)) = z(z^2-1), \quad 
B_0(z) = \frac{i}{2}(E_0(z)-E_0^\sharp(z)) = 1-2z^2. 
\]
The set of polynomials $p_0(z)=1$, $p_1(z)=z$, $p_2(z)=z^2$ 
forms a basis of $\mathcal{H}(E_0)$, but it is not orthogonal. 
Applying the determinant form of the Gram–Schmidt process, 
\[
q_k(z) 
= \frac{1}{\sqrt{D_{k-1}D_k}}
\begin{vmatrix}
\langle p_0, p_0 \rangle & \langle p_0, p_1 \rangle & \cdots & \langle p_0, p_{k-2} \rangle &  p_0 \\
\langle p_1, p_0 \rangle & \langle p_1, p_1 \rangle & \cdots & \langle p_1, p_{k-2} \rangle & p_1 \\
\vdots & \vdots & \ddots & \vdots & \vdots \\
\langle p_{k-1}, p_0 \rangle& \langle p_{k-1}, p_1\rangle & \cdots & \langle p_{k-1}, p_{k-2} \rangle & p_{k-1} 
\end{vmatrix}, 
\]
\[
D_k=
\begin{vmatrix}
\langle p_0, p_0 \rangle & \langle p_0, p_1 \rangle & \cdots & \langle p_0, p_{k-1} \rangle  \\
\langle p_1, p_0 \rangle & \langle p_1, p_1 \rangle & \cdots & \langle p_1, p_{k-1} \rangle  \\
\vdots & \vdots & \ddots & \vdots  \\
\langle p_{k-1}, p_0 \rangle& \langle p_{k-1}, p_1\rangle & \cdots & \langle p_{k-1}, p_{k-1} \rangle 
\end{vmatrix}
\]
with 
\[
\aligned 
\langle p_0,p_0 \rangle_{\mathcal{H}(E_0)} 
& 
= 2\pi, \\
\langle p_0,p_1 \rangle_{\mathcal{H}(E_0)} 
& = \langle p_1,p_0 \rangle_{\mathcal{H}(E_0)} 
= 0, \\
\langle p_0,p_2 \rangle_{\mathcal{H}(E_0)} 
& = \langle p_2,p_0 \rangle_{\mathcal{H}(E_0)} 
= \pi, \\
\langle p_1,p_1 \rangle_{\mathcal{H}(E_0)} 
& 
= \pi, \\
\langle p_1,p_2 \rangle_{\mathcal{H}(E_0)} 
& = \langle p_2,p_1 \rangle_{\mathcal{H}(E_0)} 
= 0, \\
\langle p_2,p_2 \rangle_{\mathcal{H}(E_0)} 
& 
= \pi,
\endaligned 
\]
we obtain the orthonormal basis
\[
q_0(z) = \frac{1}{\sqrt{2\pi}}, \quad 
q_1(z) = \frac{z}{\sqrt{\pi}}, \quad 
q_2(z) = \frac{2z^2-1}{\sqrt{2\pi}}. 
\]
Other natural bases will be considered in the next section.
\medskip

By \eqref{eq_701} and \eqref{eq_707}, we obtain 
\[
m_0 = 2\pi, \quad m_k = \frac{\pi}{2}(1+(-1)^k) \quad (k \geq 1), \quad H_2 = \pi^3. 
\]
Therefore, the reproducing kernel is given by 
\[
K(z,w) = -\frac{1}{\pi^3} \det
\begin{bmatrix}
0 & 1 & w & \ w^2 \\
1 & 2\pi & 0 &  \pi \\
\bar{z} & 0 & \pi &  0 \\
\bar{z}^2 & \pi & 0 &  \pi 
\end{bmatrix}
\]
by \eqref{eq_805}. 
One readily checks that this agrees with 
\[
\frac{\overline{A_0(z)}B_0(w) - A_0(w)\overline{B_0(z)}}{\pi (w - \bar{z})}
=
\frac{\overline{z(z^2-1)}(1-2w^2) - w(w^2-1)\overline{(1-2z^2)}}{\pi (w - \bar{z})}.
\]

%
%
\section{Multiplication by an independent variable in de Branges spaces}
%
%

To establish the isometric isomorphism 
$L^2(\tau_0)\simeq L^2(d\mu_{\Theta_0}) \simeq \mathcal{K}(\Theta_0)\simeq\mathcal{H}(E_0)$, 
we use an orthogonal system consisting of the eigenfunctions of 
some self-adjoint extension of the multiplication operator $\mathsf{M}$ on $\mathcal{H}(E_0)$. 

\subsection{} \label{section_9_1}

For a de Branges space $\mathcal{H}(E)$, 
we denote by ${\mathsf M}$ the multiplication operator by the independent variable 
\[
{\mathsf M}: F(z) \mapsto zF(z) 
\]
equipped with the domain
\[
{\mathfrak D}({\mathsf M}) := \{ F(z) \in {\mathcal H}(E) ~|~ z F(z) \in {\mathcal H}(E) \}.
\]

\begin{theorem} \label{thm_domain}
Let $E \in \mathcal{HB}$. 
For a real number $\theta \in [0,\pi)$, we define 
\begin{equation*}
S_\theta(z) 
= e^{i\theta}E(z)-e^{-i\theta}E^\sharp(z)
\end{equation*}
and let $A(z)$ and $B(z)$ be as in \eqref{eq_804}. 
\begin{enumerate}
\item The domain ${\mathfrak D}({\mathsf M})$ is not dense in ${\mathcal H}(E)$ 
if and only if 
\[
A(z)u+B(z)v \in {\mathcal H}(E)
\]
for some $(u,v)\in{\C}^2\setminus\{(0,0)\}$ with $|u-iv|=|u+iv|$, 
or equivalently, 
\[
S_\theta(z) \in {\mathcal H}(E) 
\]
for some $\theta \in [0,\pi)$ 
$($We have $A(z)u+B(z)v=\lambda S_\theta(z)$ for some $\lambda \in {\C} \setminus\{0\}$ and $\theta \in [0,\pi)$ 
if $|u-iv|=|u+iv|)$. 
\item The function $S_{\theta}(z)$ belongs to ${\mathcal H}(E)$ for at most one $\theta \in [0,\pi)$. 
\item Suppose that $A(z)u+B(z)v$ belongs to ${\mathcal H}(E)$ 
for some $(u,v)\in{\C}^2\setminus\{(0,0)\}$ with $|u-iv|=|u+iv|$. 
Then the orthogonal complement of ${\mathfrak D}({\mathsf M})$ 
is spanned by $A(z)u+B(z)v:$
\[
{\mathfrak D}({\mathsf M})^{\perp} = {\rm span}\{A(z)u+B(z)v\}={\rm span}\{S_\theta(z)\}. 
\]
Hence, if the domain ${\mathfrak D}({\mathsf M})$ is not dense in ${\mathcal H}(E)$, 
the closure of ${\mathfrak D}({\mathsf M})$ has codimension $1$ in ${\mathcal H}(E)$. 
\item If ${\mathfrak D}({\mathsf M})$ is not dense in ${\mathcal H}(E)$, 
${\mathfrak D}({\mathsf M})^{\perp}$ is a de Branges space of dimension one.
\end{enumerate}
\end{theorem}
\begin{proof}
See Theorem I of ~\cite{dB61}, Theorem 29 of ~\cite{dB68}, 
and also Corollary 6.3 of \cite{KaWo99}. 
\end{proof}

\begin{theorem}
Let $E \in \mathcal{HB}$. 
\begin{enumerate}
\item The operator ${\mathsf M}$ is symmetric and closed, i.e., 
$\langle {\mathsf M} F,G \rangle = \langle F, {\mathsf M} G \rangle$ for every $F,G\in{\mathfrak D}({\mathsf M})$ and 
its graph is closed in ${\mathcal H}(E) \oplus {\mathcal H}(E)$. 
\item The operator $\mathsf M$ has deficiency indices $(1,1)$. 
\item The operator $\mathsf M$ is real with respect to the involution $\sharp$, 
i.e, satisfies ${\mathsf M}(F^\sharp)=({\mathsf M}F)^\sharp$ for $F \in {\mathfrak D}({\mathsf M})$. 
\end{enumerate}
\end{theorem}
\begin{proof}
See Proposition 4.2 of \cite{KaWo99}. 
\end{proof}

\begin{theorem} \label{thm_self_adjoint}
Let $E \in \mathcal{HB}$. 
The set of all self-adjoint  extensions of ${\mathsf M}$ 
is parametrized by $\theta \in [0,\pi)$. 
If $S_\theta \not\in \mathcal{H}(E)$, 
the corresponding self-adjoint extension ${\mathsf M}_\theta$ 
is an operator described as follows: 
the domain of ${\mathsf M}_\theta$ is defined by
\begin{equation*}
{\mathfrak D}({\mathsf M}_\theta) 
= \left\{\left.
G(z) = \frac{S_\theta(w_0)F(z)-S_\theta(z)F(w_0)}{z-w_0} ~\right|~
F(z) \in {\mathcal H}(E)
\right\},
\end{equation*}
and the operation is defined by 
\begin{equation*}
{\mathsf M}_\theta G(z) 
= z \, G(z) + F(w_0)S_\theta(z),
\end{equation*}
where $w_0$ is a fixed complex number with $S_\theta(w_0)\not=0$. 
The domain ${\mathfrak D}({\mathsf M}_\theta)$ does not depend on the choice of  $w_0$. 

If $S_\theta \in \mathcal{H}(E)$, the corresponding self-adjoint extension $\mathsf{M}_\theta$ 
is given by the linear relation
\[
\mathsf{M}_\theta 
= \{ (G(z),zG(z)+wS_\theta(z))\,|\, G \in \mathfrak{D}(\mathsf{M}),~w \in \C\}. 
\]
In particular, $\mathfrak{D}(\mathsf{M}_\theta)=\mathfrak{D}(\mathsf{M})$. 
\end{theorem}
\begin{proof}
See Propositions 4.6 and 6.1 of \cite{KaWo99}. 
\end{proof}
\begin{remark}
Since $S_\theta \in \operatorname{Assoc}\mathcal{H}(E) = \mathcal{H}(E)+z\mathcal{H}(E)$, 
\[
\frac{S_\theta(w_0)F(z)-S_\theta(z)F(w_0)}{z-w_0} \in \mathcal{H}(E)
\] 
for arbitrary $w_0 \in \C$ and $F \in \mathcal{H}(E)$.  
\end{remark}

All zeros of $S_\theta(z)$ are real and they are simple (since $E(z)$ has no real zeros). 
The set of all eigenfunctions of $\mathsf{M}_\theta$ is  
\[
\frac{S_\theta(z)}{z-\gamma_\theta}, \quad S_\theta(\gamma_\theta)=0
\]
and it forms an orthogonal basis of $\mathcal{H}(E)$ 
if ${\mathfrak D}({\mathsf M}_\theta)$ is dense, that is, 
$S_\theta \not\in \mathcal{H}(E)$. 
The number of $\theta\in[0,\pi)$ for which
$S_\theta\in\mathcal H(E)$
is at most one.
If $S_\theta \in \mathcal{H}(E)$, the set 
\[
S_\theta(z), \quad \frac{S_\theta(z)}{z-\gamma_\theta}, \quad S_\theta(\gamma_\theta)=0
\]
forms an orthogonal basis of $\mathcal{H}(E)$. 
Proposition 2.1 in \cite{Re02} is often useful 
for determining whether $S_\theta \in \mathcal{H}(E)$. 
Since 
\[
S_\theta(z) = -2i A_{\theta+\pi/2}(z) = -2i\frac{E_{\theta+\pi/2}(z)+E_{\theta+\pi/2}^\sharp(z)}{2}
\]
for $E_\theta(z) =  e^{i\theta} E(z)$, 
there is no loss of generality in considering only the orthonormalization of  $A(z)/(z-\gamma)$. 
As a result, the set 
\[
\sqrt{\frac{2}{\pi|\Theta'(\gamma)|}} \frac{A(z)}{z-\gamma}, \quad A(\gamma)=0
\]
forms an orthonormal system, where $\Theta(z)=E^\sharp(z)/E(z)$ 
(\cite[Theorem 22]{dB68} and its proof. See also \cite[Proof of Proposition 3.2]{Su23b}). 
Since $A(\gamma)=0$ and
$E(\gamma)= -iB(\gamma)\neq 0$,
we have $B(\gamma)\neq0$.
Differentiating
$\Theta=(A-iB)/(A+iB)$ and evaluating at $z=\gamma$ yield
\[
\Theta'(z) = \frac{2iA'(\gamma)}{B(\gamma)}.
\] 

\subsection{} 

For $E_0(z)=z^3+2iz^2-z-i$, 
\[
S_\theta(z) 
= e^{i\theta}E_0(z)-e^{-i\theta}E_0^\sharp(z)
= 2i(z^3-z)\sin\theta -2i(1-2z^2)\cos\theta
\]
Since $E_0$ is a polynomial of degree three, 
\[
\mathcal{H}(E_0)=\C 1+\C z+\C z^2 
\]
as a set, and thus
\[
{\mathfrak D}({\mathsf M})=\C 1+\C z  \subsetneq \mathcal{H}(E_0). 
\]
Therefore, there exists only one $\theta \in [0,\pi)$ 
for which $S_\theta \in \mathcal{H}(E_0)$, 
and $S_\theta \not\in \mathcal{H}(E_0)$ for other $\theta$'s. 
In fact, 
\[
S_0(z) = -2i(1-2z^2) \in \mathcal{H}(E_0)
\]
and $S_\theta(z) \not\in \mathcal{H}(E_0)$ for $\theta \in (0,\pi)$. 

For $\theta \in (0,\pi)$,
elements of ${\mathfrak D}({\mathsf M}_\theta)$
obtained from $F(z)=1,\,z,\,z^2$ are
\[
\aligned 
\frac{S_\theta(w_0)-S_\theta(z)}{(-2i)(z-w_0)}
&= 2w_0\cos\theta-(1-w_0^2)\sin\theta+(2\cos\theta+w_0\sin\theta)z+z^2\sin\theta, \\
\frac{S_\theta(w_0)z-S_\theta(z)w_0}{(-2i)(z-w_0)}
&= \cos\theta + w_0(2\cos\theta+w_0\sin\theta)z+w_0z^2\sin\theta, \\
\frac{S_\theta(w_0)z^2-S_\theta(z)w_0^2}{(-2i)(z-w_0)}
&= w_0\cos\theta + (\cos\theta+w_0\sin\theta)z+w_0^2z^2\sin\theta. 
\endaligned 
\]
These three polynomials are linearly independent
for $\theta\in(0,\pi)$.
Since $\dim \mathcal H(E_0)=3$,
it follows that
${\mathfrak D}(\mathsf M_\theta)=\mathcal H(E_0)$. 
The function $S_\theta(z)$ is a polynomial of degree three 
with three distinct real roots $\gamma_{i,\theta}$ ($i=1,2,3$). 
(We may take $\gamma_{i,\theta}$ ($i=1,2,3$) to be continuous functions of $\theta$.) 
Therefore, the set of polynomials 
\[
\frac{S_\theta(z)}{z-\gamma_{1,\theta}}, \quad 
\frac{S_\theta(z)}{z-\gamma_{2,\theta}}, \quad 
\frac{S_\theta(z)}{z-\gamma_{3,\theta}} 
\]
forms an orthogonal basis of $\mathcal{H}(E_0)$ consisting of eigenfunctions of ${\mathsf M}_\theta$. 
These eigenfunctions are quadratic, so they do not belong to ${\mathfrak D}({\mathsf M})$. 
Since 
\[
S_{\pi/2}(z) = 2iz(z^2-1)
\]
when $\theta=\pi/2$, 
the set of functions 
\begin{equation} \label{eq_901}
F_0(z) = \frac{z^2-1}{\sqrt{\pi}}, \quad 
F_{1}(z) = \frac{z(z+1)}{\sqrt{2\pi}}, \quad 
F_{-1}(z) = \frac{z(z-1)}{\sqrt{2\pi}} 
\end{equation} 
forms 
an orthonormal basis of $\mathcal{H}(E_0)$ consisting of eigenfunctions of $M_{\pi/2}$. 

\subsection{} \label{section_9_3} 

Using the basis \eqref{eq_901}, 
we confirm that the restriction to the level set $\sigma(\Theta_0)$ 
defines an isometric isomorphism  $\mathcal{K}(\Theta_0) \simeq L^2(d\mu_{\Theta_0})$. 
First, 
\[
\aligned 
\Vert F_k \Vert_{\mathcal{H}(E_0)}^2 & 
= \Vert F_k/E_0 \Vert_{L^2(\R)}^2
= \int |F_k(z)|^2 \frac{dz}{|E_0(z)|^2} = 1 \quad (k=0,1,-1)
\endaligned
\]
by definition. On the other hand, 
\[
\aligned 
\lim_{z\to 0}\frac{F_0(z)}{E_0(z)}  = -\frac{i}{\sqrt{\pi}}, \quad 
\lim_{z\to 1}\frac{F_1(z)}{E_0(z)}  = -i\sqrt{\frac{2}{\pi}}, \quad
\lim_{z\to -1}\frac{F_{-1}(z)}{E_0(z)}  = -i\sqrt{\frac{2}{\pi}}
\endaligned 
\]
by direct calculation. Therefore, 
\[
\aligned 
\Vert F_k/E_0 \Vert_{L^2(\mu_{\Theta_0})}^2 & = 
\int \left| \frac{F_k(z)}{E_0(z)} \right|^2 d\mu_{\Theta_0}(z) = 1 \quad (k=0,1,-1)
\endaligned
\]
by \eqref{eq_707}. 

%
%
\section{One special screw line in the model space $\mathcal{K}(\Theta_0)$} 
%
%

\subsection{} We define 
\[
\widehat{\psi_0}(z) := \frac{F_0(z)}{E_0(z)} = \frac{1}{\sqrt{\pi}}\frac{1+\Theta_0(z)}{z}, 
\]
\[
\widehat{\psi_1}(z) := \frac{F_1(z)}{E_0(z)} = \frac{1}{\sqrt{2\pi}}\frac{1+\Theta_0(z)}{z-1}, 
\]
\[
\widehat{\psi_{-1}}(z) := \frac{F_{-1}(z)}{E_0(z)} = \frac{1}{\sqrt{2\pi}}\frac{1+\Theta_0(z)}{z+1},  
\]
using the basis \eqref{eq_901} of $\mathcal{H}(E_0)$. 
Then, 
$\widehat{\psi_0}(z)$, $\widehat{\psi_1}(z)$, $\widehat{\psi_{-1}}(z)$ 
all belong to $\mathcal{K}(\Theta_0) \subset L^2(\R)$ 
by $\mathcal{K}(\Theta_0)=\mathcal{H}(E_0)/E_0$ 
and are the Fourier transforms of functions of $L^2(\R)$. 
Moreover, we define
\[
\aligned 
\mathfrak{S}_t(z)
& := \sqrt{\pi}
\left(\lim_{z\to0} \frac{e^{izt}-1}{z} \right) \widehat{\psi_0}(z) 
+ \sqrt{\frac{\pi}{2}}\cdot\frac{e^{it}-1}{1} \, \widehat{\psi_1}(z)
+ \sqrt{\frac{\pi}{2}}\cdot\frac{e^{-it}-1}{(-1)} \, \widehat{\psi_{-1}}(z) \\
& = \sqrt{\pi}  it \, \widehat{\psi_0}(z) 
+ \sqrt{\frac{\pi}{2}}(e^{it}-1) \, \widehat{\psi_1}(z)
- \sqrt{\frac{\pi}{2}}(e^{-it}-1) \, \widehat{\psi_{-1}}(z).
\endaligned 
\]
Then, $\mathfrak{S}_t(z)$ belongs to $\mathcal{K}(\Theta_0)$ for every $t \in \R$. 
Furthermore, the relations
\[
\frac{1}{\pi}\langle \mathfrak{S}_t, \mathfrak{S}_s \rangle = G_{g_0}(t,s) 
\]
and 
\[
\frac{1}{2\pi}\Vert \mathfrak{S}_t \Vert_{\mathcal{K}(\Theta_0)}^2
= \frac{1}{2\pi}\Vert \mathfrak{S}_t \Vert_{L^2(\R)}^2 = - g_0(t)
\]
hold by \eqref{eq_403} and \eqref{eq_502}. 
In other words, the mapping $\mathfrak{S}:\R \to \mathcal{K}(\Theta_0) \subset L^2(\R)$ 
is a screw line associated with $g_0(t)$ up to the multiple of $\sqrt{\pi}$. 

\subsection{} Using the screw line $\mathfrak{S}$, we define 
\[
\aligned 
\widehat{\mathcal{P}_\phi}(z) 
& = \int_{-\infty}^{\infty} \phi(t)\mathfrak{S}_t(z) \, dt \\
& =  \sqrt{\pi} \widehat{\phi}'(0)\, \widehat{\psi_0}(z) 
+ \sqrt{\frac{\pi}{2}}(\widehat{\phi}(1)-\widehat{\phi}(0)) \widehat{\psi_1}(z)
- \sqrt{\frac{\pi}{2}}(\widehat{\phi}(-1)-\widehat{\phi}(0)) \widehat{\psi_{-1}}(z).
\endaligned 
\]
Then, the mapping $\phi \mapsto \widehat{\mathcal{P}_\phi}$ 
defines an isometric isomorphism from $\mathcal{H}(G_{g_0})$ to $\mathcal{K}(\Theta_0)$, 
because 
\begin{equation} \label{eq_1001}
\aligned 
\Vert \widehat{\mathcal{P}_\phi} \Vert_{\mathcal{K}(\Theta_0)}^2
& =
\Vert \widehat{\mathcal{P}_\phi} \Vert_{L^2(\R)}^2 \\
& = \pi \left|\widehat{\phi}'(0)\right|^2 
+ \frac{\pi}{2}\left|\widehat{\phi}(1)\right|^2 
+ \frac{\pi}{2}\left|\widehat{\phi}(-1)\right|^2 \\
& = \pi \langle \phi, \phi \rangle_{g_0,\infty}
\endaligned 
\end{equation} 
by \eqref{eq_504} and $\widehat{\phi}(0)=0$ for $\phi \in \mathcal{H}(G_{g_0})$. 

%
%
\section{Summary, I.} 
%
%

From the results obtained in the above sections, 
we obtain the commutative diagram
\[
\xymatrix{
\mathcal{H}(G_{g_0})  \ar[d]_-{\Phi_1} 
\ar[r]^-{\frac{1}{\sqrt{\pi}}E_0\widehat{\mathcal{P}}}\ar@{}[rd]|-{\circlearrowright}
 &
\mathcal{H}(E_0)\ar[d]^-{r_{\Theta_0}} 
\\
L^2(\tau_0) \ar[r]_-{\times \frac{1}{\sqrt{\pi}}}
 & L^2(d\mu_{\Theta_0})
}
\]
with 
\[
\Vert \phi \Vert_{g_0,\infty}^2 
= \Vert \Phi_1(\phi) \Vert_{L^2(\tau_0)}^2 
= \frac{1}{\pi} \Vert \widehat{\mathcal{P}}_\phi \Vert_{L^2(\R)}^2
= \frac{1}{\pi} \Vert \widehat{\mathcal{P}}_\phi|_{\sigma(\Theta_0)} \Vert_{L^2(d\mu_{\Theta_0})}^2,
\]
where $r_{\Theta_0}$ is the composition of 
$\mathcal{H}(E_0) \to \mathcal{K}(\Theta_0): f \mapsto f/E_0$ 
and the restriction $\mathcal{K}(\Theta_0) \to L^2(d\mu_{\Theta_0}):  
F \mapsto F|_{\sigma(\Theta_0)}$. 
All maps in this diagram are isometric isomorphisms: 
\begin{itemize}
\item $\Phi_1:\mathcal{H}(G_{g_0}) \to L^2(\tau_0)$ is in Section \ref{section_5_3}. 
\item $L^2(\tau_0) \to L^2(d\mu_{\Theta_0})$ is isometric by \eqref{eq_403} and \eqref{eq_707}. 
\item $r_\Theta:\mathcal{H}(E_0) \to \mathcal{K}(\Theta_0) 
\to L^2(d\mu_{\Theta_0})$ is in  Section \ref{section_9_3}.
\item $\tfrac{1}{\sqrt{\pi}}E_0\widehat{\mathcal{P}}:
\mathcal{H}(G_{g_0}) \to \mathcal{K}(\Theta_0) \to \mathcal{H}(E_0)$ 
is isometric by \eqref{eq_1001}.
\end{itemize}
Note that the isomorphism $r_{\Theta_0}:\mathcal{H}(E_0) \to L^2(d\mu_{\Theta_0})$ 
depends on the choice of the basis or generator of $\mathcal{H}(E_0)$. 
This choice of basis is essential for the definition of 
$\widehat{\mathcal{P}}:\mathcal{H}(G_{g_0}) \to \mathcal{K}(\Theta_0)$. 
Recall that the choice of this basis comes from the choice of self-adjoint extensions 
of the multiplication operator $\mathsf{M}$. 

The structure of a de Branges space is controlled by a Hamiltonian $H$ 
as described in the next section.  
The isometric isomorphism from $\mathcal{H}(G_{g_0})$ to $\mathcal{H}(E_0)$ 
factors through the space $\widehat{L}^2(H_0)$ 
generated by a Hamiltonian $H_0$ of $\mathcal{H}(E_0)$: 
\begin{equation} \label{eq_1101}
\xymatrix{
\mathcal{H}(G_{g_0}) 
\ar[r]^-{\mathsf{L}_0} \ar[rd]_-{\frac{1}{\sqrt{\pi}}E_0\widehat{\mathcal{P}}} 
& \widehat{L}^2(H_0) \ar[d]^-{\mathsf{W}} \ar@{}@<-2.5ex>[d]|{\circlearrowright}  \\
 & \mathcal{H}(E_0).
}
\end{equation}
Our next task is to describe this diagram in detail. 
To this end, we need the basis chosen in the definition of $\widehat{\mathcal{P}}$, 
which consists of eigenfunctions of the self-adjoint extension $\mathsf{M}_{\pi/2}$.
 
%
%
\section{Structure Hamiltonians of de Branges spaces} 
%
%

\subsection{} 
Here we refer to \cite[Section 2]{Win14}, \cite{Wo17}, and references therein 
for theoretical and historical details on Hamiltonians. 

\begin{definition}
A ${\rm Sym}_2(\R)$-valued function $H$ defined on     
$I=(t_0,t_1)$, $[t_0,t_1)$, $(t_0,t_1]$, $[t_0,t_1]$  
 $(-\infty<t_0<t_1 \leq \infty)$ is called a \textit{\textbf{Hamiltonian}} 
if 
\begin{itemize}
\item[(H1)] $H(t)$ is nonnegative definite for almost every $t \in I$, 
\item[(H2)] $H\not\equiv 0$ on any subset of $I$ with positive Lebesgue measure, 
\item[(H3)] $H$ is locally integrable on $I$ with respect to the Lebesgue measure.
\end{itemize} 
For the Hamiltonian $H$, 
the first-order system 
\begin{equation} \label{eq_1201} 
-\frac{d}{dt}
\begin{bmatrix}
A(t,z) \\ B(t,z)
\end{bmatrix}
= z 
\begin{bmatrix}
0 & -1 \\ 1 & 0
\end{bmatrix}
H(t)
\begin{bmatrix}
A(t,z) \\ B(t,z)
\end{bmatrix}, \quad 
z \in \C
\end{equation}
is called a \textit{\textbf{canonical system}} on $I$. 
\end{definition}

The Hamiltonian $H$ is said to be in the {\it limit circle case} at the left endpoint $t_0$ 
if 
\[
\int_{t_0}^t {\rm tr}\,H(u) \, du < \infty
\]
for some (and hence for all) $t \in I$ 
and in the {\it limit point case} otherwise. 
A similar notion is defined for the right endpoint $t_1$. 
For the Hamiltonian $H$ on $I$, 
an open interval $J \subset I$ is called an {\it indivisible interval} of type $\theta$ 
for $H$ if 
\[
\begin{bmatrix}
\cos\theta & \sin\theta
\end{bmatrix}
\begin{bmatrix}
0 & -1 \\ 1 & 0
\end{bmatrix}
H(t) =0
\]
holds for almost every $t \in J$, that is, 
\[
H(t) = h(t)
\begin{bmatrix} 
\cos^2\theta & \cos\theta\sin\theta \\
\cos\theta\sin\theta & \sin^2\theta
\end{bmatrix}
= h(t)
\begin{bmatrix}  \cos\theta \\ \sin\theta \end{bmatrix} 
[\,\cos\theta~~\sin\theta\,]
\]
for almost every $t \in J$ for some scalar function $h(t)$. 
A point $t \in I$ is called {\it regular} for $H$ 
if it is not inner point of an indivisible interval. 
The set of all regular points for $H$ is denoted by $I_{\rm reg}$. 
\bigskip

\begin{theorem} \label{thm_12_1}
Let $E \in \mathcal{HB}$.  
Then, there exists a Hamiltonian $H$ 
defined on some interval $I=(t_0,t_1]$ 
such that the following statements hold.

\begin{enumerate}
\item The Hamiltonian $H$ is in the limit circle case at the right endpoint $t_1$. 
\item Let ${}^t[A(t,z)~B(t,z)]$ ($t \in I$, $z\in \C$) 
be the unique solution of the initial value problem consisting of \eqref{eq_1201} and 
\[
\begin{bmatrix} A(t_1,z) \\ B(t_1,z) \end{bmatrix}
=
\begin{bmatrix} A(z) \\ B(z) \end{bmatrix},
\]
where $A$ and $B$ are functions in \eqref{eq_804}. 
Then, $E(t,z):=A(t,z)-iB(t,z)$ belongs to $\mathcal{HB}$ for all $t \in I$. 
\item The set of all de Branges subspaces coincides with $\{ \mathcal{H}(E_t)\,|\, t \in I_{\rm reg} \}$. 
\item We have 
\[
K(t; z, z) 
= \frac{\overline{A(t,z)}B(t,z) - A(t,z)\overline{B(t,z)}}{\pi (z - \bar{z})} ~\to 0 \quad 
\text{as} \quad t \to t_0
\]
pointwise for all $z \in \C$. 
\end{enumerate}
\end{theorem}
\begin{proof}
See Theorems 35 and 40 in \cite{dB68}. 
Note that (3) does not mean that $\mathcal{H}(E_t)$ 
for $t \in I \setminus I_{\rm reg}$ is not a de Branges subspace of $\mathcal{H}(E)$ 
(cf.  Section \ref{section_12_3}). 
\end{proof}

\begin{definition} 
Let $E \in \mathcal{HB}$. 
We refer to the Hamiltonian $H$ satisfying (1)--(4) in Theorem \ref{thm_12_1}  
as the \textit{\textbf{structure Hamiltonian}} of $E$. 
\end{definition}

A Hamiltonian $H_2$ on $I_2$ 
is called a reparameterization of a Hamiltonian $H_1$ on $I_1$ 
if there exists an increasing bijection $\tau$ 
between $I_2$ and $I_1$ such that $\tau$ and $\tau^{-1}$ are both absolutely continuous
and
\[
H_2(x) = H_1(\tau(x))\tau'(x). 
\]
Precisely, the structure Hamiltonian of $E \in \mathcal{HB}$ 
is the equivalence class for the reparameterization of $H$ satisfying (1)--(4). 
The Hamiltonian $H$ is uniquely determined up to reparameterization 
under the normalization $E(0)=1$ for $E$. 

\subsection{} \label{section_12_2} 

To find the structure Hamiltonian of a polynomial $E \in \mathcal{HB}$, 
we may use \cite[p. 137, Proof of Theorem VII]{dB61}. 
Let 
\[
W(z) = \begin{bmatrix} A(z) & B(z) \\ C(z) & D(z) \end{bmatrix}
\]
be a matrix consisting of polynomials of $z$ satisfying $W(0)=I$, 
\begin{equation} \label{eq_1202}
A(z)D(z)-B(z)C(z)=1, 
\end{equation}
\begin{equation} \label{eq_1203}
\Re \Bigl[ A(z)\overline{D(z)}-B(z)\overline{C(z)} \,\Bigr] \geq 1, 
\end{equation}
and
\begin{equation} \label{eq_1204}
\frac{B(z)\overline{A(z)}-A(z)\overline{B(z)}}{z-\bar{z}} \geq 0, 
\quad  \frac{D(z)\overline{C(z)}-C(z)\overline{D(z)}}{z-\bar{z}} \geq 0.
\end{equation}
If $W(z)$ has degree $r:=\max\{{\rm deg}\,A,\,{\rm deg}\,B,\,{\rm deg}\,C,\,{\rm deg}\,D\}$, 
we find that 
\begin{equation*}
W(z)=\prod_{k=1}^{r} \left(I-z\begin{bmatrix}
\alpha_k & \beta_k \\ \beta_k & \gamma_k 
\end{bmatrix} J\right)W(0), \quad 
J=\begin{bmatrix} 0 & -1 \\ 1 & 0 \end{bmatrix}
\end{equation*}
with the factors taken from left to right in numerical order, 
where 
\[
\alpha_k \geq 0, \quad \gamma_k \geq 0, \quad \det \begin{bmatrix}
\alpha_k & \beta_k \\ \beta_k & \gamma_k 
\end{bmatrix} =0,
\]
and 
\[
\alpha_k \gamma_{k-1}+\gamma_k \alpha_{k-1}-2\beta_k \beta_{k-1} >0, 
\]
for $k=2,3,\dots,r$. Recall $W(0)=I$ by hypothesis. 
Define
\begin{equation*}
t_k:=\sum_{j=1}^k (\alpha_j+\gamma_j) \quad (k=1,\dots,r), \quad t_0=0  
\end{equation*}
and the locally constant ${\rm Sym}_2(\R)$-valued function 
\[
H(t):=\frac{1}{\alpha_k+\gamma_k}
\begin{bmatrix}
\alpha_k & \beta_k \\ 
\beta_k & \gamma_k
\end{bmatrix} 
\]
for $t \in [t_{k-1}, t_k)$ ($k=1,\dots,r$). 
Then, the solution $W(t,z)$ of the initial value problem
\begin{equation} \label{eq_1205}
\frac{d}{dt} W(t,z) J
= z W(t,z) H(t) , \quad W(0,z)=I
\end{equation}
recovers the original $W(z)$ as $W(t_r,z)=W(z)$. 
The solution $W(t,z)$ is often called 
the \textbf{fundamental solution} or \textbf{fundamental matrix} for $H$. 
Since the Hamiltonian $H$ in \eqref{eq_1205} is locally constant, 
it is not difficult to solve. 
In fact, we find that the solution is determined inductively by 
\begin{equation} \label{eq_1206}
W(t,z) = W(t_{k-1},z)
\left(
I-
\frac{z(t-t_{k-1})}{\alpha_{k}+\gamma_{k}}
\begin{bmatrix} 
\alpha_k & \beta_k \\ \beta_k & \gamma_k
\end{bmatrix}
J
\right), \quad 
J=\begin{bmatrix} 0 & -1 \\ 1 & 0 \end{bmatrix}
\end{equation}
for $t_{k-1} \leq t \leq t_k$, $1 \leq k \leq r$.

\subsection{} \label{section_12_3}
Apply the way of Section \ref{section_12_2} to $E_0(z)=z^3+2iz^2-z-i$. 
We define 
$C(z):=\frac{1}{2}(E_0(z)+E_0^\sharp(z))$ 
and 
$D(z):=\frac{i}{2}(E_0(z)-E_0^\sharp(z))$. 
By $E_0^\sharp(z)=-E_0(-z)$, 
$C(z)$ is odd 
and
$D(z)$ is even.  In fact, 
\[
C(z) = z^3 -z, \qquad D(z) = 1-2z^2. 
\]
Moreover, 
\[
\frac{D(z)}{C(z)} = Q_0(z),
\]
where $Q_0(z)$ is given by \eqref{eq_604}. 
We choose 
\[
A(z) = 1-2z^2, \quad B(z) = 4z. 
\]
Then $A(z)-iB(z)$ belongs to $\mathcal{HB}$ 
as well as $E_0(z)=C(z)-iD(z)$ and 
\eqref{eq_1202}, \eqref{eq_1203}, \eqref{eq_1204} 
are satisfied. According to the decomposition,
\[
\aligned 
W_0(z)
&=\begin{bmatrix} A(z) & B(z) \\ C(z) & D(z) \end{bmatrix} 
= \begin{bmatrix} 1-2z^2 & 4z \\ z^3-z & 1-2z^2 \end{bmatrix} \\
&=
\left(I-z\begin{bmatrix} 0 & 0 \\ 0 &1/2 \end{bmatrix}J\right)
\left(I-z\begin{bmatrix} 4 & 0 \\ 0 &0  \end{bmatrix}J\right)
\left(I-z\begin{bmatrix} 0 & 0 \\ 0 &1/2 \end{bmatrix}J\right),
\endaligned 
\]
we define
\[
t_0:=0, \quad t_1 := \frac{1}{2}, \quad t_2 := \frac{1}{2} + 4, \quad t_3 := \frac{1}{2} + 4 + \frac{1}{2}
\]
and 
\begin{equation} \label{eq_1207}
H_0(t) = 
\begin{cases} 
\displaystyle{\begin{bmatrix} 0 & 0 \\ 0 & 1 \end{bmatrix}}, & t_0 \leq t < t_1, 
\quad t_2 \leq t \leq t_3 \\[16pt] 
\displaystyle{\begin{bmatrix} 1 & 0 \\ 0 & 0 \end{bmatrix}}, & t_1 \leq t < t_2. 
\end{cases}
\end{equation}
Then, clearly $H_0$ is a Hamiltonian on $I_0=[t_0,t_3]=[0,5]$. 
Intervals $(t_0,t_1)=(0,1/2)$ and $(t_2,t_3)=(9/2,5)$ are indivisible intervals of type $\pi/2$ 
and the interval $(t_1,t_2)=(1/2,9/2)$ is the indivisible interval of type $0$ for $H_0$. 
The set of all regular points $I_{0,{\rm reg}}$ is $\{t_0,t_1,t_2,t_3\}$ 
and both endpoints $t_0$ and $t_3$ are limit circle case.  
\medskip

From \eqref{eq_1206}, we find that the matrix 
\[
W(t,z)=\begin{bmatrix} A(t,z) & B(t,z) \\ C(t,z) & D(t,z) \end{bmatrix} 
\]
with 
\[
A(t,z) = 
\begin{cases} 
~1, & t_0 \leq t < t_1, \\[8pt] 
~1, & t_1 \leq t < t_2, \\[8pt]   
~1+(18-4t)z^2, & t_2 \leq t \leq t_3, \\  
\end{cases}
\]
\[
B(t,z) = 
\begin{cases} 
~0, & t_0 \leq t < t_1, \\[8pt] 
~\displaystyle{\left(t-\frac{1}{2}\right)z}, & t_1 \leq t < t_2, \\[8pt]   
~4z, & t_2 \leq t \leq t_3, \\  
\end{cases}
\]
\[
C(t,z) = 
\begin{cases} 
~-tz, & t_0 \leq t < t_1, \\[8pt] 
~\displaystyle{-\frac{1}{2}z}, & t_1 \leq t < t_2, \\[8pt]   
~(2t-9)z^3-(t-4)z, & t_2 \leq t \leq t_3, \\  
\end{cases}
\]
\[
D(t,z) = 
\begin{cases} 
~1, & t_0 \leq t < t_1, \\[8pt] 
~\displaystyle{1+\frac{1}{4}(1-2t)z^2}, & t_1 \leq t < t_2, \\[8pt]   
~1-2z^2 & t_2 \leq t \leq t_3
\end{cases}
\]
solves the initial value problem \eqref{eq_1205} for $H=H_0$ 
and recovers $W_0(z)$ at $t=t_3$: 
\[
\begin{bmatrix} A(t_3,z) & B(t_3,z) \\ C(t_3,z) & D(t_3,z) \end{bmatrix} 
= \begin{bmatrix} 1 - 2z^2 & 4z \\ z^3-z & 1-2z^2 \end{bmatrix}.
\]
Let $E_0(t,z):=C(t,z) - iD(t,z)$. Then, all de Branges subspaces of 
$\mathcal{H}(E_0)$ are 
\[
\aligned 
\mathcal{H}(E(t_3,z)) & = \mathcal{H}(E_0(z)) \simeq \C^3, \\
\mathcal{H}(E(t_2,z)) & = \mathcal{H}(-z/2+(2z^2-1)) \simeq \C^2, \\
\mathcal{H}(E(t_1,z)) & = \mathcal{H}(-z/2+i) = \C, \\
\mathcal{H}(E(t_0,z)) & = \mathcal{H}(1) = \{0\}.
\endaligned 
\]

%
%
\section{$\widehat{L}^2(H)$: Hilbert spaces associated with Hamiltonians } 
%
%

\subsection{} 
Here we refer to \cite[Section 2]{Win14} and references therein 
for theoretical details on Hilbert spaces $\widehat{L}^2(H)$.   
For a Hamiltonian $H$ on $[0,L)$, 
we denote by $\mathcal{L}(H)$ 
the linear space of column vectors 
$\mathfrak{F}(t)={}^{t}[\,f(t)~~g(t)\,]$ 
of compactly supported continuous functions on $[0,L)$ 
satisfying
\[
\Vert \mathfrak{F} \Vert_{H}^2 = 
\frac{1}{\pi}
\int_{0}^{L} [\,f(t)~~g(t)\,]\,H(t) 
\begin{bmatrix} 
\overline{f(t)} \\ \overline{g(t)}
\end{bmatrix} dt < \infty 
\]
and define 
\[
\mathcal{L}^\circ(H) :=\{ \mathfrak{F} \in \mathcal{L}(H)\,|\, \Vert \mathfrak{F} \Vert_H=0\}. 
\]
Then the quotient space $\mathcal{L}(H)/\mathcal{L}^\circ(H)$ forms 
a pre-Hilbert space by the norm 
\[
\Vert \mathfrak{F} + \mathcal{L}^\circ(H) \Vert_H := \Vert \mathfrak{F} \Vert_H. 
\]
We denote by $L^2(H)$ the Hilbert space 
obtained by completing $\mathcal{L}(H)/\mathcal{L}^\circ(H)$. 
\medskip

When $H$ has indivisible intervals, 
we denote by $\widehat{L}^2(H)$ the subspace of $L^2(H)$ 
consisting of all equivalence classes represented by vectors $\mathfrak{F}(t)={}^{t}[\,f(t)~~g(t)\,]$  such that 
\[
[\,\cos\theta~~\sin\theta\,]
\begin{bmatrix}  f(t) \\ g(t) \end{bmatrix} = c_{I,\theta,\mathfrak{F}}
\]
holds for some $c_{I,\theta,\mathfrak{F}} \in \C$ 
on every indivisible $I$ of type $\theta$. 
For $\mathfrak{F} \in\widehat{L}^2(H)$, 
\[
\Vert \mathfrak{F} \Vert_{H}^2 = \sum_{I} |c_{I,\theta,\mathfrak{F}}|^2 + \frac{1}{\pi}
\int_{[0,L) \setminus \left( \bigcup_I I \right)} [\,f(t)~~g(t)\,]\,H(t) 
\begin{bmatrix} 
\overline{f(t)} \\ \overline{g(t)}
\end{bmatrix} dt.
\]
Clearly $L^2(H)=\widehat{L}^2(H)$ if $H$ has no indivisible intervals. 

\subsection{} 

For the Hamiltonian $H_0$ of \eqref{eq_1207}, 
the equivalence class of $0$ in $L^2(H_0)$ consists of vectors of the form 
\[
\mathfrak{F}^\circ(t) = 
\begin{cases} 
\displaystyle{\begin{bmatrix} f_1(t) \\ 0 \end{bmatrix}}, & t_0 \leq t < t_1, \\[16pt] 
\displaystyle{\begin{bmatrix} 0 \\ g_2(t) \end{bmatrix}}, & t_1 \leq t < t_2, \\[16pt]   
\displaystyle{\begin{bmatrix} f_3(t) \\ 0 \end{bmatrix}}, & t_2 \leq t < t_3,  
\end{cases}
\]
and the subspace $\widehat{L}^2(H_0)$ of  $L^2(H_0)$ 
consists of equivalence classes represented by vectors of the form 
\begin{equation} \label{eq_1301}
\mathfrak{F}(t) = 
\begin{cases} 
\displaystyle{\begin{bmatrix} f_1(t) \\ c_1 \end{bmatrix}}, & t_0 \leq t < t_1, \\[16pt] 
\displaystyle{\begin{bmatrix} c_2 \\ g_2(t) \end{bmatrix}}, & t_1 \leq t < t_2, \\[16pt]   
\displaystyle{\begin{bmatrix} f_3(t) \\ c_3 \end{bmatrix}}, & t_2 \leq t < t_3, \\  
\end{cases}
\end{equation}
where $c_1$, $c_2$, $c_3$ are constants 
and 
$f_1(t)$, $g_2(t)$, $f_3(t)$ are measurable functions. 
In particular, $\widehat{L}^2(H_0)$ is three-dimensional, 
since every equivalence class is uniquely determined by the constants $c_1$, $c_2$, $c_3$. 

%
%
\section{Weyl transform/de Branges transform} 
%
%

\subsection{} For the structure Hamiltonian $H$ of $E \in \mathcal{HB}$, 
we obtain 
\begin{equation*} 
K(z,w) 
= \frac{\overline{A(z)}B(w) - A(w)\overline{B(z)}}{\pi (w - \bar{z})}
= 
\frac{1}{\pi}
\int_{t_0}^{t_1} [\,A(t,w)~~B(t,w)\,]\,H(t) 
\begin{bmatrix} A(t,\bar{z}) \\ B(t,\bar{z}) \end{bmatrix} dt
\end{equation*}
from \eqref{eq_1201} in a way similar to the proof of Lemma 2.1 in \cite{Dym70}. 
The correspondence 
\[
\begin{bmatrix} A(\cdot,\bar{z}) \\ B(\cdot,\bar{z}) \end{bmatrix} ~\mapsto~ K(z,\cdot) 
\]
extends to the unitary map $\mathsf{W}:\widehat{L}^2(H) \to \mathcal{H}(E)$ defined by
\[
(\mathsf{W} \mathfrak{F})(z)
= 
\frac{1}{\pi}
\int_{t_0}^{t_1} [\,A(t,z)~~B(t,z)\,]\,H(t) 
\mathfrak{F}(t) \, dt. 
\]
This is an analogue of the Fourier transform called 
the \textit{\textbf{Weyl transform}} or the \textit{\textbf{de Branges transform}} 
(cf. Appendix~\ref{section_app_d}). 
Using an orthonormal basis $\{F_\lambda\}_{\lambda \in \Lambda}$ 
of $\mathcal{H}(E)$, the inverse transform can be written as  
\[
\aligned 
(\mathsf{W}^{-1}F)(t)
& = \int_{\R} F(z) \sum_{\lambda \in \Lambda} C(\lambda) 
\overline{F_\lambda(z)} 
\begin{bmatrix} A(t,z) \\ B(t,z) \end{bmatrix} \, \frac{dz}{|E(z)|^2} 
\endaligned 
\]
for suitable coefficients $\{C(\lambda)\}_{\lambda \in \Lambda}$ 
such that $\mathsf{W} \circ \mathsf{W}^{-1}={\rm Id}_{\mathcal{H}(E)}$ holds. 
Alternatively, using the isomorphism $r_\Theta:\mathcal{H}(E) \to L^2(d\mu_{\Theta})$, 
\[
\aligned 
(\mathsf{W}^{-1}F)(t)
& = \int_{-\infty}^{\infty} 
F(\lambda) \begin{bmatrix} A(t,\lambda) \\ B(t,\lambda) \end{bmatrix} 
d \mu_{\Theta}(\lambda)
\endaligned 
\]
\cite[p. 533]{Win14}. 

\subsection{} 

The integral formula defining the Weyl transform
extends naturally to a map $\mathsf{W}: L^2(H_0) \to \mathcal{H}(E_0)$ as follows 
\[
\aligned 
(\mathsf{W}\mathfrak{F})(z)
& = \frac{1}{\pi} \int_{0}^{5} [C(s,z)~D(s,z)] \,H_0(s) 
\begin{bmatrix} f_1(s) \\ f_2(s) \end{bmatrix} \, ds \\
& = 
\frac{1}{\pi}\left(
\int_{0}^{1/2} f_2(s) \,ds 
-\frac{z}{2} \int_{1/2}^{1/2+4} f_1(s) \,ds 
+(1-2z^2) \int_{1/2+4}^{1/2+4+1/2} f_2(s) \,ds
\right).
\endaligned
\]
The images of the vectors
\[
\begin{bmatrix} C(t,0) \\ D(t,0) \end{bmatrix},~
\begin{bmatrix} C(t,1) \\ D(t,1) \end{bmatrix}, ~
\begin{bmatrix} C(t,-1) \\ D(t,-1) \end{bmatrix} 
\]
are 
\begin{equation*}
\aligned 
\mathsf{W}\begin{bmatrix} C(t,0) \\ D(t,0) \end{bmatrix} (z)
& = - \frac{1}{\sqrt{\pi}} F_0(z), \\
\mathsf{W}\begin{bmatrix} C(t,1) \\ D(t,1) \end{bmatrix} (z)
& = \sqrt{\frac{2}{\pi}} F_1(z) , \\
\mathsf{W}\begin{bmatrix} C(t,-1) \\ D(t,-1) \end{bmatrix} (z)
& = \sqrt{\frac{2}{\pi}} F_{-1}(z).
\endaligned 
\end{equation*}
Therefore, the image of $\mathsf{W}$ is three-dimensional 
and hence surjective, but not injective. 
Moreover, the above shows that the restriction of $\mathsf{W}$ to $\widehat{L}^2(H_0)$ is bijective.  
\smallskip

The inverse transform $\mathsf{W}^{-1}: \mathcal{H}(E_0) \to \widehat{L}^2(H_0)$ is given by 
\[
\aligned 
(\mathsf{W}^{-1}F)(t) 
& = \int_{-\infty}^{\infty} 
F(\gamma) \begin{bmatrix} C(t,\gamma) \\ D(t,\gamma) \end{bmatrix} 
d \mu_{\Theta_0}(\gamma) \\
& = 
\pi 
\left( F(0) \begin{bmatrix} C(t,0) \\ D(t,0) \end{bmatrix} 
+ \frac{F(1)}{2} \begin{bmatrix} C(t,1) \\ D(t,1) \end{bmatrix} 
+ \frac{F(-1)}{2} \begin{bmatrix} C(t,-1) \\ D(t,-1) \end{bmatrix} 
\right). 
\endaligned 
\]
The image of the orthonormal basis $\{F_0,F_1,F_{-1}\}$ of $\mathcal{H}(E_0)$ 
in \eqref{eq_901} is
\begin{equation} \label{eq_1401}
\aligned 
\mathfrak{F}_0(t)
& = (\mathsf{W}^{-1}F_0)(t)
 = - \sqrt{\pi} \begin{bmatrix} C(t,0) \\ D(t,0) \end{bmatrix}, \\
\mathfrak{F}_1(t)
& = (\mathsf{W}^{-1}F_1)(t)
 = \sqrt{\frac{\pi}{2}}  \begin{bmatrix} C(t,1) \\ D(t,1) \end{bmatrix}, \\
\mathfrak{F}_{-1}(t)
& = (\mathsf{W}^{-1}F_{-1})(t)
 = \sqrt{\frac{\pi}{2}}  \begin{bmatrix} C(t,-1) \\ D(t,-1) \end{bmatrix}.
\endaligned 
\end{equation}
We can directly confirm that 
$\{\mathfrak{F}_0, \, \mathfrak{F}_1, \, \mathfrak{F}_{-1} \}$ 
is an orthonormal basis of $\widehat{L}^2(H_0)$ 
and $\mathsf{W}(\mathfrak{F}_k)=F_k$ for $k=0,1,-1$.  
\medskip

Let $\mathfrak{f}_1$, $\mathfrak{f}_2$, $\mathfrak{f}_3$ 
be vectors in $\widehat{L}^2(H_0)$ 
corresponding to the constants $(c_1,c_2,c_3)=(1,0,0),\,(0,1,0),\,(0,0,1)$, respectively, in \eqref{eq_1301}. 
Then, they form a basis of $\widehat{L}^2(H_0)$ and 
\[
(\mathsf{W}\,\mathfrak{f}_1)(z)=\frac{1}{2\pi}, \quad 
(\mathsf{W}\,\mathfrak{f}_2)(z)=\frac{-2z}{\pi}, \quad 
(\mathsf{W}\,\mathfrak{f}_3)(z)=\frac{1-2z^2}{2\pi}. 
\]

%
%
\section{Choice of basis for $\mathcal{H}(G_g)$} 
%
%

\subsection{} 

We already have a basis in $L^2(\tau_0)$ 
corresponding to eigenfunctions \eqref{eq_901} of 
a self-adjoint extension $\mathsf{M}_{\pi/2}$ on $\mathcal{H}(E_0)$ by the following series of isometries: 
\[
\left( \widehat{L}^2(H_0)~\overset{\mathsf{W}}{\longrightarrow}\right)~
\mathcal{H}(E_0)~\overset{\times E^{-1}}{\longrightarrow}~
\mathcal{K}(\Theta_0)~\overset{\rm rest.}{\longrightarrow}~
L^2(d\mu_\Theta) ~\overset{\times \sqrt{\pi}}{\longrightarrow}~ 
L^2(\tau_0). 
\]
To define the map $\mathsf{L}_0:\mathcal{H}(G_{g_0}) \to \widehat{L}^2(H_0)$ 
that forms part of the commutative diagram \eqref{eq_1101}, 
we choose a basis of $\mathcal{H}(G_{g_0})$ such that 
the image of 
\[
\mathcal{H}(G_{g_0})~\overset{\Phi_1}{\longrightarrow}~
L^2(\tau_0)
\]
matches the image of eigenfunctions of $\mathsf{M}_{\pi/2}$ above. 

\subsection{} \label{section_15_2}

Let $\alpha_1$, $\alpha_2$, $\alpha_3$ be the roots of $E_0(z)$: 
\[
E_0(z)=z^3+2iz^2-z-i=(z-\alpha_1)(z-\alpha_2)(z-\alpha_3),
\] 
\[
\{\alpha_1,\alpha_2, \alpha_3\} = \{ -i \cdot 1.75488...,~ \pm 0.744862 - i \cdot 0.122561... \}.
\]
Then, we obtain the partial fraction decompositions 
\[
\aligned 
\frac{z^2-1}{E_0(z)} 
& =
\frac{\alpha_1^2-1}{(\alpha_1-\alpha_2)(\alpha_1-\alpha_3)}\cdot
\frac{1}{z-\alpha_1} \\
& \qquad \quad 
+\frac{\alpha_2^2-1}{(\alpha_2-\alpha_1)(\alpha_2-\alpha_3)}\cdot
\frac{1}{z-\alpha_2} \\
& \qquad \qquad \quad 
+\frac{\alpha_3^2-1}{(\alpha_3-\alpha_1)(\alpha_3-\alpha_2)}\cdot
\frac{1}{z-\alpha_3},
\endaligned 
\]
\[
\aligned 
\frac{z^2+z}{E_0(z)} 
& =
\frac{\alpha_1^2+\alpha_1}{(\alpha_1-\alpha_2)(\alpha_1-\alpha_3)}\cdot
\frac{1}{z-\alpha_1} \\
& \qquad \quad 
+\frac{\alpha_2^2+\alpha_2}{(\alpha_2-\alpha_1)(\alpha_2-\alpha_3)}\cdot
\frac{1}{z-\alpha_2} \\
& \qquad \qquad \quad 
+\frac{\alpha_3^2+\alpha_3}{(\alpha_3-\alpha_1)(\alpha_3-\alpha_2)}\cdot
\frac{1}{z-\alpha_3},
\endaligned 
\]
\[
\aligned 
\frac{z^2-z}{E_0(z)} 
& =
\frac{\alpha_1^2-\alpha_1}{(\alpha_1-\alpha_2)(\alpha_1-\alpha_3)}\cdot
\frac{1}{z-\alpha_1} \\
& \qquad \quad 
+\frac{\alpha_2^2-\alpha_2}{(\alpha_2-\alpha_1)(\alpha_2-\alpha_3)}\cdot
\frac{1}{z-\alpha_2} \\
& \qquad \qquad \quad 
+\frac{\alpha_3^2-\alpha_3}{(\alpha_3-\alpha_1)(\alpha_3-\alpha_2)}\cdot
\frac{1}{z-\alpha_3}.
\endaligned 
\]
In view of these decompositions, 
we define $\psi_0(t)$, $\psi_1(t)$, and $\psi_{-1}(t)$ by 
\[
\aligned 
\psi_0(t)
& = (-i)\left[
\frac{\alpha_1^2-1}{(\alpha_1-\alpha_2)(\alpha_1-\alpha_3)}\cdot
\exp(-i \alpha_1 t) 
\right. \\
& \qquad \qquad \quad 
+\frac{\alpha_2^2-1}{(\alpha_2-\alpha_1)(\alpha_2-\alpha_3)}\cdot
\exp(-i \alpha_2 t) \\
& \qquad \qquad \qquad \quad \left.
+\frac{\alpha_3^2-1}{(\alpha_3-\alpha_1)(\alpha_3-\alpha_2)}\cdot
\exp(-i \alpha_3 t)\right], 
\endaligned 
\]
\[
\aligned 
\psi_1(t)
& = \frac{(-i)}{\sqrt{2}}\left[
\frac{\alpha_1^2+\alpha_1}{(\alpha_1-\alpha_2)(\alpha_1-\alpha_3)}\cdot
\exp(-i \alpha_1 t) \right.\\
& \qquad \qquad \quad 
+\frac{\alpha_2^2+\alpha_2}{(\alpha_2-\alpha_1)(\alpha_2-\alpha_3)}\cdot
\exp(-i \alpha_2 t) \\
& \qquad \qquad \qquad \quad \left.
+\frac{\alpha_3^2+\alpha_3}{(\alpha_3-\alpha_1)(\alpha_3-\alpha_2)}\cdot
\exp(-i \alpha_3 t)
\right],
\endaligned 
\]
\[
\aligned 
\psi_{-1}(t)
& = \frac{(-i)}{\sqrt{2}}\left[
\frac{\alpha_1^2-\alpha_1}{(\alpha_1-\alpha_2)(\alpha_1-\alpha_3)}\cdot
\exp(-i \alpha_1 t) \right.\\
& \qquad \qquad \quad 
+\frac{\alpha_2^2-\alpha_2}{(\alpha_2-\alpha_1)(\alpha_2-\alpha_3)}\cdot
\exp(-i \alpha_2 t) \\
& \qquad \qquad \qquad \quad \left.
+\frac{\alpha_3^2-\alpha_3}{(\alpha_3-\alpha_1)(\alpha_3-\alpha_2)}\cdot
\exp(-i \alpha_3 t)
\right]
\endaligned 
\]
for nonnegative $t$, 
and $\psi_0(t)=\psi_1(t)=\psi_{-1}(t)=0$ for negative $t$. 
Then, $\psi_0$, $\psi_1$, and $\psi_{-1}$ all belong to 
$L^2(0,\infty) \cap L^1(0,\infty) \cap C^\infty((0,\infty))$ and 
\[
\widehat{\psi_0}(z) = \frac{1+\Theta_0(z)}{z}, 
\quad 
\widehat{\psi_1}(z) = \frac{1}{\sqrt{2}}\frac{1+\Theta_0(z)}{z-1}, 
\quad
\widehat{\psi_{-1}}(z) = \frac{1}{\sqrt{2}}\frac{1+\Theta_0(z)}{z+1}. 
\]
Therefore, if we define $\phi_k(t):=i\psi_k^\prime(t)$ ($k=0,1,-1$),   
\[
\aligned 
\phi_0(t)
& = \delta_0(t) \\ 
& \quad +\left[
\frac{\alpha_1^2-1}{(\alpha_1-\alpha_2)(\alpha_1-\alpha_3)}\cdot
(-i\alpha_1)\exp(-i \alpha_1 t) \right.\\
& \qquad \qquad \quad \quad \left.
+\frac{\alpha_2^2-1}{(\alpha_2-\alpha_1)(\alpha_2-\alpha_3)}\cdot
(-i\alpha_2)\exp(-i \alpha_2 t)\right. \\
& \qquad \qquad \qquad \qquad \left.
+\frac{\alpha_3^2-1}{(\alpha_3-\alpha_1)(\alpha_3-\alpha_2)}\cdot
(-i\alpha_3)\exp(-i \alpha_3 t)\right], 
\endaligned 
\]
\[
\aligned 
\phi_1(t)
& = \frac{1}{\sqrt{2}}\,\delta_0(t) \\
& \quad+ \frac{1}{\sqrt{2}}\left[
\frac{\alpha_1^2+\alpha_1}{(\alpha_1-\alpha_2)(\alpha_1-\alpha_3)}\cdot
(-i\alpha_1)\exp(-i \alpha_1 t) \right.\\
& \qquad \qquad \quad \quad \left.
+\frac{\alpha_2^2+\alpha_2}{(\alpha_2-\alpha_1)(\alpha_2-\alpha_3)}\cdot
(-i\alpha_2)\exp(-i \alpha_2 t) \right.\\
& \qquad \qquad \qquad \qquad \left.
+\frac{\alpha_3^2+\alpha_3}{(\alpha_3-\alpha_1)(\alpha_3-\alpha_2)}\cdot
(-i\alpha_3)\exp(-i \alpha_3 t)
\right], 
\endaligned 
\]
\[
\aligned 
\phi_{-1}(t)
& = \frac{1}{\sqrt{2}}\,\delta_0(t) \\
& \quad + \frac{1}{\sqrt{2}}\left[
\frac{\alpha_1^2-\alpha_1}{(\alpha_1-\alpha_2)(\alpha_1-\alpha_3)}\cdot
(-i\alpha_1)\exp(-i \alpha_1 t) \right.\\
& \qquad \qquad \quad \quad \left.
+\frac{\alpha_2^2-\alpha_2}{(\alpha_2-\alpha_1)(\alpha_2-\alpha_3)}\cdot
(-i\alpha_2)\exp(-i \alpha_2 t)\right.\\
& \qquad \qquad \qquad \qquad \left.
+\frac{\alpha_3^2-\alpha_3}{(\alpha_3-\alpha_1)(\alpha_3-\alpha_2)}\cdot
(-i\alpha_3)\exp(-i \alpha_3 t)
\right],
\endaligned 
\]
then 
\[
\widehat{\phi_0}(z) = z\cdot\frac{1+\Theta(z)}{z}, 
\quad 
\widehat{\phi_1}(z) = z\cdot\frac{1}{\sqrt{2}}\frac{1+\Theta(z)}{z-1}, 
\quad
\widehat{\phi_{-1}}(z) = z\cdot\frac{1}{\sqrt{2}} \frac{1+\Theta(z)}{z+1}
\]
holds. Hence, 
\[
\Phi_1(\phi_0,z) = \sqrt{\pi}\,\frac{F_0(z)}{E_0(z)}, 
\quad 
\Phi_1(\phi_1,z) = \sqrt{\pi}\,\frac{F_1(z)}{E_0(z)}, 
\quad 
\Phi_1(\phi_{-1},z) = \sqrt{\pi}\,\frac{F_{-1}(z)}{E_0(z)}.
\]

\subsection{} 

We define the linear map $\mathsf{L}_0:\mathcal{H}(G_{g_0}) \to \widehat{L}^2(H_0)$ 
so that the basis $\{\phi_0,\,\phi_1,\,\phi_{-1}\}$ of $\mathcal{H}(G_g)$ 
in Section \ref{section_15_2} 
corresponds to the basis
\[
\left\{
\kappa_0 \begin{bmatrix} C(t,0) \\ D(t,0) \end{bmatrix},~
\kappa_1 \begin{bmatrix} C(t,1) \\ D(t,1) \end{bmatrix},~
\kappa_{-1} \begin{bmatrix} C(t,-1) \\ D(t,-1) \end{bmatrix} 
\right\}
\]
of $\widehat{L}^2(H_0)$ isometrically for some constants $\kappa_0$, $\kappa_1$, and $\kappa_{-1}$. 
To this end, we set 
\begin{equation} \label{eq_1501}
\aligned 
L_0(t,s)
& = \pi  \, \left\{ -\left(\lim_{z\to0} \frac{e^{izs}-1}{z} \right) 
\begin{bmatrix} C(t,0) \\ D(t,0) \end{bmatrix} \right. \\
& \qquad \qquad \left.
+ \frac{1}{2}\cdot\frac{e^{is}-1}{1} 
\begin{bmatrix} C(t,1) \\ D(t,1) \end{bmatrix}
+ \frac{1}{2}\cdot\frac{e^{-is}-1}{(-1)} 
\begin{bmatrix} C(t,-1) \\ D(t,-1) \end{bmatrix}
\right\} \\
& = \pi  \, \left\{ -is \begin{bmatrix} C(t,0) \\ D(t,0) \end{bmatrix} 
+ \frac{1}{2}(e^{is}-1)
\begin{bmatrix} C(t,1) \\ D(t,1) \end{bmatrix}
- \frac{1}{2}(e^{-is}-1)
\begin{bmatrix} C(t,-1) \\ D(t,-1) \end{bmatrix}
\right\} . 
\endaligned 
\end{equation}
Using the kernel $L_0(t,s)$, we define 
\[
\mathsf{L}_0: \mathcal{H}(G_{g_0})~\to~\widehat{L}^2(H_0), 
\quad 
(\mathsf{L}_0\phi)(t) := \int_{-\infty}^{\infty} L_0(t,s) \phi(s) \, ds. 
\]
Then, it is isometric and the image of the basis $\{\phi_0,\,\phi_1,\,\phi_{-1}\}$ is 
\[
\aligned
(\mathsf{L}_0 \phi_0)(t) 
& = \sqrt{\pi}\begin{bmatrix} C(t,0) \\ D(t,0) \end{bmatrix} = -(\mathsf{W}^{-1}F_0)(t), \\
(\mathsf{L}_0 \phi_1)(t)
& = \sqrt{\frac{\pi}{2}}\begin{bmatrix} C(t,1) \\ D(t,1) \end{bmatrix} =(\mathsf{W}^{-1}F_1)(t), \\
(\mathsf{L}_0 \phi_{-1})(t) 
& = \sqrt{\frac{\pi}{2}}\begin{bmatrix} C(t,-1) \\ D(t,-1) \end{bmatrix}=(\mathsf{W}^{-1}F_{-1})(t)
\endaligned 
\]
by \eqref{eq_1401}.
By definition \eqref{eq_1501}, 
\[
\aligned 
(\mathsf{L}_0 \phi)(t)
& = \pi  \, \left( -\widehat{\phi}'(0) \begin{bmatrix} C(t,0) \\ D(t,0) \end{bmatrix} 
+ \frac{1}{2}\widehat{\phi}(1)
\begin{bmatrix} C(t,1) \\ D(t,1) \end{bmatrix}
- \frac{1}{2} \widehat{\phi}(-1)
\begin{bmatrix} C(t,-1) \\ D(t,-1) \end{bmatrix}
\right)
\endaligned 
\]
for $\phi \in \mathcal{H}(G_g)$ (since $\widehat{\phi}(0)=0$), and therefore
\[
\aligned 
(\mathsf{W}\mathsf{L}_0 \phi)(z)
& = \sqrt{\pi}\widehat{\phi}'(0) F_0(z)
+ \sqrt{\frac{\pi}{2}} \widehat{\phi}(1)
F_1(z)
- \sqrt{\frac{\pi}{2}} \widehat{\phi}(-1)
F_{-1}(z) \\
& = E(z)\widehat{\mathcal{P}_\phi}(z)
\endaligned 
\]
by \eqref{eq_1401}. Hence, we obtain \eqref{eq_1101}. 
\medskip

On the other hand, 
\[
\langle \phi_k,\,\phi_l \rangle_{G_g} = \frac{1}{\pi} \,\delta_{k,l}, \quad k,l\in\{0,1,-1\}
\]
holds by \eqref{eq_501} and \eqref{eq_702}, and therefore
\[
J(t,s) = \pi (\phi_0(t) \overline{\phi_0(s)}
+ \phi_1(t) \overline{\phi_1(s)}
+ \phi_{-1}(t)\overline{\phi_{-1}(s)} )
\]
provides the reproducing kernel of $\mathcal{H}(G_g)$. 

%
%
\section{Summary, II.} \label{section_16}
%
%

We have now obtained the following commutative diagram consisting entirely of isometric isomorphisms: 
\[
\xymatrix{
&&\widehat{L}^2(H_0)\ar[rrdd]^-{\mathsf{W}}&&\\\\
\mathcal{H}(G_{g_0}) \ar[rruu]^-{\mathsf{L}_0}
\ar[dd]_-{\Phi_1}
\ar[rrrr]^-{\frac{1}{\sqrt{\pi}}E_0\widehat{\mathcal{P}}}
\ar@{}[rrrruu]|{\circlearrowright}
\ar@{}[rrrrdd]|{\circlearrowright}&&&&\mathcal{H}(E_0)\simeq\mathcal{K}(\Theta_0) 
\ar[dd]^-{r_{\Theta_0}} \\
&&&&\\
L^2(\tau_0) \ar[rrrr]_-{\times \frac{1}{\sqrt{\pi}}}&&&& L^2(d\mu_{\Theta_0}) 
}
\]
In addition,
 we also obtain the following correspondence of the quantities 
to define the objects of the above commutative diagram: 
\[
\xymatrix{
&&H_0 &&\\\\
g_0 \ar@{-->}[rruu]^-{}
\ar[dd]_-{\text{formula \eqref{eq_401}}}
\ar[rrrr]^-{\text{formula \eqref{eq_602}}}
&&&& 
Q_0 \Leftrightarrow \Theta_0 \Leftrightarrow E_0
\ar[dd]^-{\text{formula \eqref{eq_701}}} 
\ar[lluu]_-{\text{\qquad structure Hamiltonian}}
\\
&&&&\\
\tau_0 \ar[rrrr]_-{\times \pi}&&&& \mu_{\Theta_0}
}
\]
A direct correspondence between the screw function $g_0$ 
and the structure Hamiltonian $H_0$ does not seem to be available in any simple form.

%
%
\section{Applications to analytic number theory} \label{section_17}
%
%

The original motivation for preparing this article was the author's work \cite{Su23a}, 
which proposed an application of screw functions to the study of the Riemann zeta-function $\zeta(s)$. 
The simplest screw discussed throughout this article may be viewed as a toy model 
for the analytic structures arising from the screw function $g_\zeta$ 
associated with the Riemann zeta-function.
The explicit calculations presented in the previous sections 
provided a useful guide for formulating and investigating
the corresponding analytic structures attached to $g_\zeta$. 
Since screw functions and their associated analytic theories 
are not widely known in analytic number theory, 
it seemed useful to provide a concrete introduction 
through the simplest screw and its associated screw function. Rather than emphasizing abstract generalities, 
the aim was to illustrate the theory by explicitly 
computing the analytic objects attached to this fundamental example. 

Since the first version of this article was written, 
several developments have taken place in this direction. 
We conclude this article with a brief overview of applications of 
screw functions to analytic number theory and related topics.
\medskip

The Riemann zeta-function occupies a central position in analytic number theory.
Among the many problems surrounding it, the Riemann Hypothesis (RH),
which asserts that every nontrivial zero of $\zeta(s)$ lies on the line
$\Re(s)=1/2$, remains one of the most famous open problems in mathematics.
The connection with the Riemann zeta-function begins with the function
\begin{equation} \label{EQ_103}
\aligned 
g_\zeta(t) 
& = -4(e^{t/2}+e^{-t/2}-2) + \sum_{n \leq \exp(|t|)} \frac{\Lambda(n)}{\sqrt{n}}(|t|-\log n)  \\
& \quad - \frac{|t|}{2}( \psi(1/4) - \log \pi )
- \frac{1}{4}\left( \Phi(1,2,1/4) - e^{-|t|/2}\Phi(e^{-2|t|},2,1/4) \right),  
\endaligned 
\end{equation}
where $\Lambda(n)$ denotes the von Mangoldt function, 
defined as $\Lambda(n)=\log p$ if $n=p^k$ for a prime number $p$ with $k \in \Z_{>0}$,  
and $\Lambda(n)=0$ otherwise,  
$\psi(s)$ is the digamma function 
and $\Phi(z,s,a) = \sum_{n=0}^{\infty} (n+a)^{-s}z^n$ 
is the Hurwitz--Lerch zeta function. 

As shown in \cite[Theorem 1.2]{Su23a}, 
the function $g_\zeta$ in \eqref{EQ_103} is a screw function on $\R$ 
if and only if RH holds. 
The significance of the function $g_\zeta$ from the viewpoint of the present article is that 
it plays a role analogous to that of the screw function $g_0$ associated with the simplest screw. 
The constructions developed throughout the previous sections therefore suggest 
a collection of analytic objects naturally attached to $g_\zeta$, 
including spectral measures, reproducing kernel Hilbert spaces, 
de Branges spaces, and canonical systems.

Assuming RH, these objects fit naturally into a framework closely parallel to that 
developed for $g_0$ in the previous sections. 
Remarkably, many of them can be constructed directly from the explicit formula \eqref{EQ_103}, 
even without assuming RH. This observation has led to a number of new analytic structures 
associated with $\zeta(s)$ and has opened several new directions of investigation.

The starting point of this line of research was \cite{Su23a}, 
where it was proved that $g_\zeta$ is a screw function on $\R$ if and only if RH holds. 
That paper also showed that, for every $a>0$, 
the compression $G_{g_\zeta}[a]$ of $G_{g_\zeta}$ to the interval $(-a,a)$ 
has no zero eigenvalue if and only if RH holds. 
In addition, it was shown that $G_{g_\zeta}[a]$ is of trace class, 
and several properties of the (weighted) moments of $g_\zeta$ were investigated.
\medskip

The discussion of the previous sections suggests that, 
assuming RH, the commutative diagram of Section~\ref{section_16} 
should admit the following analogue associated with $g_\zeta$:
\[
\xymatrix{
&&\widehat{L}^2(H_\zeta)\ar[rrdd]^-{\mathsf{W}}&&\\\\
\mathcal{H}(G_{g_\zeta}) \ar[rruu]^-{\mathsf{L}_\zeta}
\ar[dd]_-{\Phi_1}
\ar[rrrr]^-{\frac{1}{\sqrt{\pi}}E_\zeta\widehat{\mathcal{P}}}
\ar@{}[rrrruu]|{\circlearrowright}
\ar@{}[rrrrdd]|{\circlearrowright}&&&&\mathcal{H}(E_\zeta)\simeq\mathcal{K}(\Theta_\zeta) 
\ar[dd]^-{r_{\Theta_\zeta}} \\
&&&&\\
L^2(\tau_\zeta) \ar[rrrr]_-{\times \frac{1}{\sqrt{\pi}}}&&&& L^2(d\mu_{\Theta_\zeta}) 
}.
\] 
Subsequent work has established most parts of this diagram, 
with the principal remaining gap being the construction of
\[
\mathsf{L}_\zeta :
\mathcal H(G_{g_\zeta})
~\to~
\widehat L^2(H_\zeta).
\]
It should be noted, however, that some issues concerning the construction of $\widehat L^2(H_\zeta)$ itself still remain to be clarified.
On the other hand, 
once the spaces $\mathcal H(G_{g_\zeta})$ and $\mathcal H(E_\zeta)$ are constructed, 
the maps $\Phi_1$ and $r_{\Theta_\zeta}$ arise naturally from the general theory. 
Thus, the essential problems are the construction of
\[
\text{(i)}~\mathcal H(G_{g_\zeta})
~\to~
\mathcal H(E_\zeta)
\quad \text{and} \quad 
\text{(ii)}~\widehat L^2(H_\zeta)
~\to~
\mathcal H(E_\zeta).
\]
Strictly speaking, the construction of (i) also relies on the intermediate spaces $L^2(\tau_\zeta)$ and $L^2(d\mu_{\Theta_\zeta})$. Nevertheless, once these spaces are available, the principal task is to construct the map between $\mathcal H(G_{g_\zeta})$ and $\mathcal H(E_\zeta)$. 
The first of these was achieved in \cite{Su26a}, while the second remains incomplete at present.

The map in (i) was constructed in \cite{Su26a}.
More precisely, let
\[
\xi(s):=s(s-1)\pi^{-s/2}\Gamma(s/2)\zeta(s),
\]
and define
\[
E_\zeta(z)
:=
\xi(1/2-iz)+\xi'(1/2-iz).
\]
The main result of \cite{Su26a} gives an isomorphism between
the de Branges space $\mathcal H(E_\zeta)$
and a Hilbert space $\mathcal H_W$
defined as the completion of $C_c^\infty(\R)$
with respect to a sesquilinear form
$\langle\cdot,\cdot\rangle_W$
satisfying
\[
\langle \phi,\phi\rangle_{g_\zeta,\infty}
=
\langle \psi,\psi\rangle_W,
\quad
\phi=\psi',
\quad
\psi\in C_c^\infty(\R).
\]
By construction,
$\mathcal H_W$ is canonically isometrically isomorphic to 
$\mathcal H(G_{g_\zeta})$.
Motivated by this correspondence,
the variational problem associated with the quadratic form
\[
\langle \phi,\phi\rangle_{g_\zeta,a}
\]
on $(-a,a)$, defined by \eqref{eq_501},
was investigated in \cite{Su26b}
without assuming RH.

The map in (ii) is expected to arise by applying the method developed in \cite{Su25a} 
for constructing the structure Hamiltonian of $\mathcal{K}(\Theta)$ 
from a meromorphic inner function $\Theta$ to
\[
\Theta_\zeta(z)=\frac{E_\zeta^\sharp(z)}{E_\zeta(z)}.
\]
However, it has not yet been verified whether $\Theta_\zeta$ satisfies 
the conditions required in \cite{Su25a}.
\medskip

The developments described above concern
the Hilbert-space structures suggested by the analogue
of the commutative diagram of Section~\ref{section_16}.
There are, however, several other directions in which
screw functions have found applications in analytic number theory.

In \cite{NaSu23a},
it was shown that RH is equivalent to the assertion that
$\exp(g_\zeta(t))$
is the characteristic function of an infinitely divisible distribution
(cf.\ Appendix~\ref{section_app_b}).
This observation is based on the representation
\[
g_\zeta(t) = \sum_\gamma \frac{e^{-i\gamma t}-1}{\gamma^2}
\]
which, under RH, is precisely the
L{\'e}vy--Khintchine formula for an infinitely divisible distribution.

The viewpoint developed for the Riemann zeta-function
is not restricted to a single zeta-function.
In \cite{Su25b},
the screw function $g_\zeta$
was generalized to screw functions associated with
Dirichlet series in the Selberg class,
and analogues of several results for $g_\zeta$
were proposed in this broader setting.
Furthermore, \cite{MaSu26}
investigated connections between screw functions
and certain finite sums arising in the study of
Goldbach-type problems.

Although many aspects of the picture remain incomplete,
the developments described above suggest that
screw functions provide a useful framework
for connecting harmonic analysis,
operator theory,
probability theory,
and analytic number theory.
The simplest screw discussed throughout this article
thus continues to serve as a useful prototype
for a broader theory linking these subjects.

The author hopes that the present exposition
will provide a convenient point of entry
to these developments for readers from both analysis
and analytic number theory.

\bigskip

\appendix

%
%
\section{Kre\u{\i}n's string} 
%
%

\subsection{} 
Here we refer to Kaltenb\"{a}ck--Winkler--Woracek 
\cite[Section 2]{KWW07}, 
Kotani--Watanabe \cite{KoWa82}, 
Langer--Winkler \cite{LaWin98}, 
and references therein. 
Kre\u{\i}n's string $S[m, L]$ consists of 
$0<L \leq \infty$ 
and a 
right-continuous 
nondecreasing 
nonnegative function $m(x)$ on $[0,L)$ such that $m(0_-)=0$. 
For $S[m, L]$, we take solutions $\phi(x,z)$ and $\psi(x,z)$ 
of the string equation 
\[
dy'(x) + z y(x) dm(x)=0
\] 
on $[0,L)$ satisfying the initial condition 
\[
\phi'(0,z)=\psi(0,z)=0, \quad \phi(0,z)=\psi'(0,z)=1,
\] 
more precisely, 
\begin{equation} \label{eq_A01}
\aligned 
\phi(x,\lambda)&=1-\lambda \int_{0}^{x} (x-y)\phi(y,\lambda) \,dm(y), \\
\psi(x,\lambda)&=x-\lambda \int_{0}^{x} (x-y)\psi(y,\lambda) \,dm(y).
\endaligned 
\end{equation}
Then the Titchmarsh--Weyl function 
\[
q(z) = \lim_{x \to L} \frac{\psi(x,z)}{\phi(x,z)}
\] 
exists and belongs to the subclass $\mathcal{N}_S$ 
of the Nevanlinna class $\mathcal{N}$ consisting of $q(z)$ such that 
\begin{equation} \label{eq_A02}
q(z) = b + \int_{0}^{\infty} \frac{d\sigma(\lambda)}{\lambda-z} 
\end{equation}
for some $b \geq 0$ and a measure $\sigma$ on $[0,\infty)$ 
with $\int_{0}^{\infty} d\sigma(\lambda)/(\lambda+1) < \infty$. 
Then,  
\begin{equation} \label{eq_A03}
b=q(-\infty)=\inf\,{\rm supp}\,(dm), 
\quad 
L=q(0_{-})=b+\int_{0}^{\infty} \frac{d\sigma(\lambda)}{\lambda}. 
\end{equation}
Kre\u{\i}n \cite{Kr52} proved that the correspondence 
$S[m,L] \to \mathcal{N}_S$ is bijective. 

\subsection{} 

For $Q_0(z)$ in \eqref{eq_604}, 
we have 
\[
q_0(z) = \frac{Q_0(\sqrt{z})}{\sqrt{z}} = \frac{1}{0-z} + \frac{1}{1-z} 
= \cfrac{1}{(-1/2)z+\cfrac{1}{4+\cfrac{1}{(-1/2)z}}}.
\]
This shows that $q_0$ belongs to $\mathcal{N}_S$, $b=0$, and $L=\infty$ 
by \eqref{eq_A02} and \eqref{eq_A03}. 
The above continued fraction expansion implies that 
the measure corresponding to $q_0$ satisfies
\[
dm_0(x) = \frac{1}{2}\delta(x) + \frac{1}{2}\delta(x-4)
\]
by \cite[Example 1.1]{KoWa82}. 
Therefore, Kre\u{\i}n's string corresponding to $q_0$ is determined by the mass distribution function 
\[
m_0(x) =
\begin{cases}
~0, & x <0, \\
~1/2, & 0<x <4, \\
~1, & x >4
\end{cases}
\]
on $[0,\infty)$ and  has mass only at points $x=0$ and $x=4$. 
The string equation \eqref{eq_A01} for $S[m_0,\infty]$ 
can be written as 
\[
\aligned 
\phi(x,\lambda)
&=0 \quad (x<0), \\ 
&=1 \quad (x=0), \\ 
&=1-\frac{1}{2}\lambda (x-0) \phi(0,\lambda) \quad (0<x<4) \\
&=1-\frac{1}{2}\lambda (x-0) \phi(0,\lambda) -\frac{1}{2}\lambda (x-4) \phi(4,\lambda) \quad (x \geq 4), \\
\psi(x,\lambda)
&=0 \quad (x \leq 0), \\
&=x -\frac{1}{2}\lambda(x-0) \psi(0,\lambda) \quad (0<x<4) \\
&=x -\frac{1}{2}\lambda(x-0) \psi(0,\lambda) -\frac{1}{2}\lambda(x-4) \psi(4,\lambda) \quad (x \geq 4). 
\endaligned 
\]
Using the initial condition $\phi(0,\lambda)=1$ and $\psi(0,\lambda)=0$, 
these equations can be solved for $\phi(4,\lambda)$ and $\psi(4,\lambda)$, 
yielding 
$\phi(4,\lambda)=1-2\lambda$, 
$\psi(4,\lambda)=4$, and 
\[
\aligned 
\phi(x,\lambda) 
&= 
\begin{cases}
~0, & x< 0, \\
~1, & x=0, \\
~ \displaystyle{1-\frac{1}{2}\,\lambda x}, & 0<x <4, \\
~ \displaystyle{\lambda(\lambda-1)\,x -4\lambda^2 +2\lambda+ 1}, & x \geq 4, 
\end{cases} \\
\psi(x,\lambda) 
&= 
\begin{cases}
~0, & x \leq  0, \\
~x, & 0<x<4, \\
~\displaystyle{(1-2\lambda)\,x+8\lambda }, & x \geq 4.
\end{cases}
\endaligned 
\]
The defining limit
\[
\lim_{x \to L} \frac{\psi(x,\lambda)}{\phi(x,\lambda)}
=
q_0(\lambda) \quad (L=\infty)
\]
is easily confirmed from the above explicit expression for solutions. 

%
%
\section{Infinitely divisible distributions} \label{section_app_b}
%
%

\subsection{} 
The integral representation of screw functions \eqref{eq_401}--\eqref{eq_402} 
makes clear the relation between screw functions and infinitely divisible distributions.
A distribution (or probability measure) $\mu$ on $\R$ is infinitely divisible 
if there exists a distribution $\mu_n$ on $\R$ 
such that 
$\mu=\mu_n \ast \dots \ast \mu_n$ ($n$-fold) for every positive integer $n$. 
For any infinitely divisible distribution $\mu$, 
there exists a triplet $(a,b,\nu)$ 
consisting of $a \in \R_{\geq 0}$, $b \in \R$ 
and a measure $\nu$ on $\R$ 
such that the characteristic function 
\begin{equation*}
\widehat{\mu}(t)=\int_{-\infty}^{\infty} e^{itx} \mu(dx)
\end{equation*}
has the L{\'e}vy--Khintchine formula 
\begin{equation} \label{LK_1}
\widehat{\mu}(t) = \exp\left[
- \frac{1}{2} a t^2 + i b t 
+ \int_{-\infty}^{\infty}\left(
e^{it\lambda}  - 1 - \frac{it\lambda}{1+\lambda^2} 
\right) \nu(d \lambda)
\right], 
\end{equation}
\begin{equation} \label{LK_2}
\nu(\{0\})=0, \quad \int_{-\infty}^{\infty} \min(1,\lambda^2) \,\nu(d\lambda) < \infty
\end{equation}
\cite[Theorem 8.1 and Remark 8.4]{Sa99}. 
The measure $\nu$ is referred to as the L{\'e}vy measure for $\mu$. 
If the L{\'e}vy measure $\nu$ satisfies $\int_{|\lambda|\leq 1} |\lambda|\,\nu(d\lambda)<\infty$, 
then \eqref{LK_1} can be rewritten as
\begin{equation} \label{LK_3}
\widehat{\mu}(t) = \exp\left[
- \frac{1}{2} a t^2 + i b_0 t 
+ \int_{-\infty}^{\infty}(e^{it\lambda}  - 1) \, \nu(d \lambda)
\right]
\end{equation}
\cite[(8.7)]{Sa99}.  
From the comparison of \eqref{eq_401}--\eqref{eq_402} 
and \eqref{LK_1}--\eqref{LK_2}, 
the relation between screw functions and infinitely divisible distributions is immediate.

\subsection{} 

Equality \eqref{eq_303} gives the L{\'e}vy--Khintchine formula 
\[
\aligned 
\exp(g_0(t))
&=\exp\left(-\frac{t^2}{2}+\frac{1}{2}(e^{it}-1)+\frac{1}{2}(e^{-it}-1)\right) \\
&=
\exp\left[
-\frac{a}{2}t^2 + \int_{-\infty}^{\infty}
( e^{it\gamma} - 1)\nu_0(d\gamma)
\right] 
\endaligned 
\]
with the L{\'e}vy measure 
\[
\nu_0 = 
 \frac{1}{2}\, \delta_1
+ \frac{1}{2}\, \delta_{-1}. 
\]
Hence, there exists a distribution $\mu_0$ such that 
$\exp(g_0(t))=\widehat{\mu_0}(t)$ by \cite[Theorem 8.1 (iii)]{Sa99}. 
Such a distribution can be described explicitly as follows.
\medskip
 
Characteristic functions of the normal distribution (or Gaussian distribution) 
$N(\mu,\sigma)$ with mean $\mu$ and standard deviation $\sigma$ 
and the Poisson distribution ${\rm Pois}(\lambda)$ with parameter $\lambda>0$
are 
\[
\exp\left( -\frac{\sigma^2t^2}{2}+i\mu t \right) 
= \int_{-\infty}^{\infty} \frac{1}{\sigma\sqrt{2\pi}}\,
\exp\left( -\frac{1}{2}\left( \frac{x-\mu}{\sigma} \right)^2 
\right) \,e^{ixt}\,dx, 
\] 
and 
\[
\exp\Bigl( \lambda(e^{it}-1) \Bigr) 
= \int_{-\infty}^{\infty} 
\left( e^{-\lambda} \sum_{k=0}^{\infty} \frac{\lambda^k}{k!}\delta_k(x)
\right) \,e^{ixt}\,dx, 
\]
respectively. 
Since 
\[
\exp(g_0(t))=\exp(-t^2/2)\exp((e^{it}-1)/2)\exp((e^{-it}-1)/2), 
\]
the corresponding distribution is the convolution of
$N(0,1)$, ${\rm Pois}(1/2)$, and the reflected Poisson distribution
$-{\rm Pois}(1/2)$. 
Therefore, 
by taking their convolution, 
\[
\aligned 
\exp(g_0(t))
&= \int_{-\infty}^{\infty} 
\left( 
\frac{1}{e\sqrt{2\pi}}
\sum_{k=0}^{\infty}\sum_{\ell=0}^{\infty} 
\frac{2^{-k-\ell}}{k!\ell!}
\exp\left( -\frac{1}{2}(x-k+\ell)^2\right)
\right) \,e^{ixt}\,dx.
\endaligned 
\]

%
%
\section{Mean periodic functions} 
%
%

\subsection{}
A screw function having a discrete spectral measure $\tau$ 
is sometimes viewed as a mean periodic function. 
Mean periodic functions are generalizations of periodic functions, 
different from almost periodic functions, 
and also generalizations of solutions of 
ordinary differential equations with constant coefficients. 
See \cite[Section 2]{FRS12} and references therein for details on the theory of mean periodicity.

Let $X$ be a locally convex separated topological $\C$-vector space 
consisting of functions on $\R$ such that the
Hahn-Banach theorem is applicable. 
Denote by $X^\ast$ 
the topological dual space of $X$ for some specified topology. 
For $f \in X$, we denote by $\mathcal{T}(f)$ the closure of the $\C$-vector space 
spanned by all translations $\{f(\cdot-r)\,|\,r \in \R\}$. 
Then, $f \in X$ is called  {\it $X$-mean-periodic} 
if $\mathcal{T}(f) \not =X$. 
If the convolution 
\[
(f \ast \varphi)(t) = \int_{-\infty}^{\infty} f(t-r) \varphi(r) \, dr
\]
is well-defined for $f \in X$ and $\varphi \in X^\ast$, 
the $X$-mean-periodicity of $f \in X$ 
is equivalent to $f \ast \varphi =0$
for some nonzero $\varphi \in X^\ast$. 
Such a convolution equation may be regarded as a generalized differential equation. 
It is said that the {\it spectral synthesis} holds in $X$ if 
\[
\mathcal{T}(f)
= \overline{{\rm span}_\C \{ P(t)e^{i\lambda t} \in \mathcal{T}(f),~\lambda \in \C\} }
\]
for any $f \in X$ satisfying $\mathcal{T}(f) \not= X$. 
When the spectral synthesis holds in $X$, 
$f \in X$ is $X$-mean-periodic if and only if 
$f$ is a limit of a sum of exponential polynomials belonging to $\mathcal{T}(f)$ 
with respect to the topology of $X$. 
Such a sum of exponential polynomials may be viewed as
the integral representation \eqref{eq_401}
with a discrete measure $\tau$.

\subsection{} 

Formula \eqref{eq_303} shows that $g_0(t)$ is a solution of the ordinary differential equation
\begin{equation} \label{eq_C01}
\left(i \frac{d}{dt}\right)^3\left[ \left(i \frac{d}{dt}\right)^2-1 \right] f(t) =0. 
\end{equation}
Therefore, for
\begin{equation} \label{eq_C02}
\varphi(t)
:=\left(i \frac{d}{dt}\right)^3\left[ \left(i \frac{d}{dt}\right)^2-1 \right] e^{-t^2}
= -4i t(8t^4-38t^2+27)e^{-t^2}, 
\end{equation}
we have
\[
(g_0 \ast \varphi)(t) = \int g_0(t-u)\varphi(u) \, du \equiv 0 \quad (t \in \R).
\]
Since $\varphi$ decays faster than any exponentials, 
the Fourier integral formula 
\[
\widehat{\varphi}(z) = \int_{-\infty}^{\infty} \varphi(t) e^{izt} \, dt = \sqrt{\pi}e^{-z^2/4}z^3(z^2-1). 
\]
holds for every $z \in \C$. Moreover, the Fourier integral formula
\[
\widehat{(g_0^+ \ast \varphi)}(z) = - \widehat{(g_0^- \ast \varphi)}(z) = -i \sqrt{\pi} (1-2z^2)e^{-z^2/4}
\]
also holds for every $z \in \C$ by taking
\[
g_0^+(t) = 
\begin{cases}
\, g_0(t), & t \geq 0,  \\ 
\, 0, & t <0, 
\end{cases}
\qquad 
g_0^-(t) = 
\begin{cases}
\, 0, & t > 0,  \\ 
\, g_0(t), & t \leq 0, 
\end{cases}
\]
because 
\[
(g_0^+ \ast \varphi)(t) = -(g_0^- \ast \varphi)(t)=-i (8t^2-3)e^{-t^2}. 
\]
Then the Fourier–Carleman transform $\mathsf{FC}(g_0)$ 
is calculated as in 
\[
\mathsf{FC}(g_0)(z) 
= \frac{\widehat{(g_0^+ \ast \varphi)}(z)}{\widehat{\varphi}(z)} 
= - \frac{\widehat{(g_0^- \ast \varphi)}(z)}{\widehat{\varphi}(z)}
= -\frac{i}{z^2} \cdot \frac{1-2z^2}{z(z^2-1)}
\]
and holds for all $z \in \C$. 
In particular, $\mathsf{FC}(g_0)(z)=\mathsf{F}(g_0)(z)$ for $z \in \C_+$. 
The above is irrelevant to which function space $g_0$ belongs. 
However, by properly choosing the space $X$ to which $g_0$ belongs, 
we can say that it is $X$-mean periodic. 
\medskip

Let $\mathcal{E}:=C^\infty(\R)$ be the space of smooth functions on $\R$ 
with the compact uniform convergence topology. 
Then, the dual space $\mathcal{E}^\ast$ 
is the space of all distributions with compact support. 
Clearly, $g_0$ belongs to $\mathcal{E}$. 
The space $\mathcal{T}(g_0)$ generated by translations of $g_0$ 
is spanned by exponential polynomials $t^ne^{\lambda t}$ 
belonging to $\mathcal{T}(g_0)$, 
because the spectral synthesis holds for $C^\infty(\R)$ 
by \cite{Sch47}. 
By \eqref{eq_303}, only five exponential polynomials 
$1$, $t$, $t^2$, $e^{it}$, $e^{-it}$ belong to $\mathcal{T}(g_0)$. 
Therefore, 
\[
\mathcal{T}(g_0) = \C \oplus \C t \oplus \C t^2 \oplus \C e^{it} \oplus \C e^{-it} ~\not=~ \mathcal{E}.
\]
Hence $g_0$ is $\mathcal{E}$-mean-periodic and 
a compactly supported distribution
annihilating $\mathcal{T}(g_0)$
can also be constructed explicitly. 
The space $\mathcal{T}(g_0)$ is nothing but the space formed 
by all the solutions of \eqref{eq_C01}. 

Let $C_{\exp}^\infty(\R)$ be the space of all smooth functions on $\R$ 
such that all their derivatives have at most exponential growth at $\pm \infty$. 
Then, all of the above discussion for $\mathcal E$ remains valid. 
In particular, the spectral synthesis holds in $C_{\exp}^\infty(\R)$ by \cite[Theorem A]{Gil70}. 
In this case, $\varphi$ of \eqref{eq_C02} belongs to $C_{\exp}^\infty(\R)^\ast$ 
and annihilates $\mathcal{T}(g_0)$. 

%
%
\section{The case of Paley--Wiener spaces} \label{section_app_d}
%
%

If we start with the Hermite--Biehler class $\mathcal{HB}$, 
the simplest example is
\[
E_r(z):=\exp(-riz)
\] 
for a positive real number $r$.  
The de Branges space generated by $E_r$ is 
nothing but the Paley--Wiener space, 
the space of Fourier transforms of square-integrable functions on $[-r,r]$: 
\[
\mathcal{H}(E_r)={\rm PW}_r = \mathsf{F}(L^2(-r,r)). 
\]
De Branges spaces were originally introduced as a generalization of Paley--Wiener spaces. 
For $E=E_r$, 
\[
A(z) =A_r(z) = \cos(rz), \quad B(z) =B_r(z) = \sin(rz), 
\]
and the reproducing kernel is 
\[
K(z,w)=K_r(z,w) 
= \frac{\overline{\cos(rz)}\sin(rw) - \cos(rw)\overline{\sin(rz)}}{\pi(w-\bar{z})}
= \frac{\sin(r(w-\bar{z}))}{\pi(w-\bar{z})}.
\]
Clearly, $S_\theta = e^{i\theta}E_r - e^{-i\theta}E_r^\sharp = -2i \sin(rz-\theta)$ does not belong to ${\rm PW}_r$. 
Therefore, the multiplication operator $\mathsf{M}$ is densely defined, and 
\[
F_n(z) := \frac{i \cos(rz)}{\sqrt{\pi r}(z-\frac{\pi}{2r}(2n-1))}, \quad n \in \Z
\]
forms the orthonormal basis of  ${\rm PW}_r$ consisting of 
eigenfunctions of the self-adjoint extension $\mathsf{M}_{\pi/2}$. 
Let $H_r(t):=I$ be the constant matrix-valued function on $I=[0,r]$. 
Then, ${}^{t}[\,\cos(tz)~\sin(tz)\,]$ is the solution of the canonical system 
\eqref{eq_1201} for $H=H_r$ satisfying the conditions of Theorem \ref{thm_12_1}. 
Therefore, $H_r$ is a structure Hamiltonian of $E_r$. 
Further, we see that  
\[
W_r(t,z) := \begin{bmatrix} \cos(tz) & \sin(tz) \\ -\sin(tz) & \cos(tz) \end{bmatrix}
\]
solves \eqref{eq_1205}, that is, this is the fundamental solution for $H_r$. 

For a vector $\mathfrak{F}(t)={}^{t}[\,f(t)~g(t)\,]$ of $L^2(H_r)=\widehat{L}^2(H_r)$, 
its Weyl transform is 
\[
(\mathsf{W} \mathfrak{F})(z)
= 
\frac{1}{\pi}
\int_{0}^{r} (f(t)\cos(tz) + g(t)\sin(tz)) \, dt. 
\]
We set 
$f(-t)=f(t)$ and $g(-t)=-g(t)$ for negative $t$ and 
\[
\Psi_{\mathfrak{F}}(t) :=  f(t) - ig(t) \quad \text{for} \quad t \in [-r,r]. 
\]
Then, 
\[
(\mathsf{W} \mathfrak{F})(z)
= 
\frac{1}{2\pi}
\int_{-r}^{r} \Psi_{\mathfrak{F}}(t)\,e^{izt} \, dt = \frac{1}{2\pi} \widehat{\Psi_{\mathfrak{F}}}(z), 
\]
that is, $\mathsf{W}$ is essentially the Fourier transform. 
To make this correspondence more explicit, note that $\mathfrak{F} \mapsto \Psi_{\mathfrak{F}}$ 
gives an isometric isomorphism $L^2(H_r) \simeq L^2(-r,r)$ (up to a constant multiple), 
because 
\[
\Vert \mathfrak{F} \Vert_{L^2(H_r)}^2 =  \frac{1}{2} \Vert \Psi_{\mathfrak{F}} \Vert_{L^2(-r,r)}^2. 
\]

On the other hand,  for $E=E_r$, 
$\Theta(z)=\Theta_r(z) = \exp(2riz) \in \mathcal{MI}$, 
\[
\sigma(\Theta_r)
= \left\{\left. \frac{\pi}{2r}(2n-1) ~\right|~ n \in \Z\right\}, 
\quad 
\mu_{\Theta_r}(\gamma)=\frac{2\pi}{|\Theta_r^\prime(\gamma)|}= \frac{\pi}{r} \quad 
\text{for $\gamma \in \sigma(\Theta_r)$}. 
\] 
From this and 
\[
\lim_{z \to \frac{\pi}{2r}(2n-1)} 
\frac{F_n(z)}{E_r(z)} = \sqrt{\frac{r}{\pi}}, 
\quad 
\lim_{z \to \frac{\pi}{2r}(2m-1)} 
\frac{F_n(z)}{E_r(z)} = 0 ~(m \not=n), 
\]
we confirm that 
\[
\Vert F \Vert_{{\rm PW}_r} = \left\Vert \left.(F/E_r) \right|_{\sigma(\Theta_r)} \right\Vert_{L^2(d\mu_{\Theta_r})}. 
\]

Moreover, for $E=E_r$, 
\[
Q(z) = Q_r(z) = i \frac{1-\Theta_r(z)}{1+\Theta_r(z)} = \tan(rz) \in \mathcal{N}. 
\]
The partial fraction expansion
\[
\tan(rz) = \lim_{N \to \infty} \frac{1}{r}\sum_{|n| \leq N} \frac{1}{\frac{\pi}{2r}(2n-1)-z}
\]
shows that formula \eqref{eq_601} holds for $a=b=0$ and 
\[
d\tau_r(\gamma) := \frac{d\mu_{\Theta_r}(\gamma)}{\pi}. 
\]
The above $Q_r$ clearly satisfies \eqref{eq_603}. Therefore, 
there exists $g_r \in \mathcal{G}_\infty$ such that
\[
\int_{0}^{\infty} g_r(t) \, e^{izt} \, dt := -\frac{i}{z^2} \tan(rz), \quad \Im(z)>h
\]
holds for some $h \geq 0$ by Theorem \ref{thm_6_1}. 
Using the partial fraction expansion of $\tan(rz)$, we find that 
\[
g_r(t) = \frac{2}{r} \sum_{n=1}^{\infty} \frac{1}{\Bigl(\frac{\pi}{2r}(2n-1)\Bigr)^2} 
\left[ \cos\left(\frac{\pi t}{2r}(2n-1)\right) -1\right]
\]
satisfies the desired equation for $h=0$. 
By definition, we see that $g_r$ satisfies \eqref{eq_401} 
for $g_r(0)=0$, $c=0$, and $d\tau_r$ above. 
Therefore, the norm of $\mathcal{H}(G_{g_r})$ has the form  
\[
\Vert \phi \Vert_{g,r}^2 = \frac{1}{r} \sum_{n \in \Z}
\frac{1}{\Bigl(\frac{\pi}{2r}(2n-1)\Bigr)^2} 
 \left| \widehat{\phi}\left( \frac{\pi}{2r}(2n-1) \right) \right|^2
\]
by \eqref{eq_502}. 
The screw line $\mathfrak{S}:\mathcal{H}(G_{g_r}) \to \mathcal{K}(\Theta_r)$ 
for $g_r$ is given by 
\[
\mathfrak{S}_t(z) 
=  \sqrt{\frac{\pi}{r}}\sum_{n \in \Z}
\frac{\exp(it \frac{\pi}{2r}(2n-1))-1}{\frac{\pi}{2r}(2n-1)} \cdot 
\frac{i(1+\Theta_r(z))}{\sqrt{\pi r}(z-\frac{\pi}{2r}(2n-1))} ,
\]
thus, an isometric isomorphism $\mathcal{H}(G_{g_r}) \to {\rm PW}_r$ is given by 
\[
\frac{1}{\sqrt{\pi}}E_r\widehat{\mathcal{P}}_\phi(z) 
= \frac{1}{\sqrt{r}} \sum_{n \in \Z}
\frac{\widehat{\phi}(\frac{\pi}{2r}(2n-1))}{\frac{\pi}{2r}(2n-1)} \cdot 
\frac{i\cos(rz)}{\sqrt{\pi r}(z-\frac{\pi}{2r}(2n-1))} ,
\]
for $\phi \in \mathcal{H}(G_{g_r})$. The map $\mathsf{L}_r: \mathcal{H}(G_{g_r}) \to L^2(H_r)$ 
satisfying $\mathsf{W}\mathsf{L}_r=\frac{1}{\sqrt{\pi}}E_r\widehat{\mathcal{P}}$ 
is given by 
\[
(\mathsf{L}_r \phi)(t) = \int_{-\infty}^{\infty}L_r(t,s) \phi(s) \, ds
\]
with the kernel
\[
L_r(t,s) = \frac{i\sqrt{\pi}}{r}\sum_{n \in \Z} (-1)^n
\frac{\exp(is \frac{\pi}{2r}(2n-1))-1}{\frac{\pi}{2r}(2n-1)} 
\begin{bmatrix}
\cos(\frac{\pi t}{2r}(2n-1)) \\ \sin(\frac{\pi t}{2r}(2n-1))
\end{bmatrix}. 
\]
As a result, we obtain the following diagram 
\[
\xymatrix{
&&L^2(H_r)\simeq \frac{1}{\sqrt{2}} L^2(-r,r)\ar[rrdd]^-{\qquad \mathsf{W}=\text{Fourier transform}}&&\\\\
\mathcal{H}(G_{g_r}) \ar[rruu]^-{\mathsf{L}_r}
\ar[dd]_-{\Phi_1}
\ar[rrrr]^-{\frac{1}{\sqrt{\pi}}E_r\widehat{\mathcal{P}}}
\ar@{}[rrrruu]|{\circlearrowright}
\ar@{}[rrrrdd]|{\circlearrowright}&&&&{\rm PW}_r\simeq\mathcal{K}(\Theta_r) 
\ar[dd]^-{r_{\Theta_r}} \\
&&&&\\
L^2(\tau_r) \ar[rrrr]_-{\times \frac{1}{\sqrt{\pi}}}&&&& L^2(d\mu_{\Theta_r}) 
}.
\]

%

\bigskip 

\noindent
Department of Mathematics, \\
School of Science, \\ 
Institute of Science Tokyo \\
2-12-1 Ookayama, Meguro-ku, \\
Tokyo 152-8551, Japan  \\[2pt]
Email: {\tt msuzuki@math.sci.isct.ac.jp}

\end{document}